\documentclass{amsart}
\usepackage{amsmath,amssymb,amsthm}
\usepackage[latin1]{inputenc}
\usepackage{tabularx,multicol}
\usepackage{graphicx,float,psfrag,array}
 \usepackage{stmaryrd}
\usepackage{url}
\usepackage{color,import}
\newcommand{\reff}[1]{(\ref{#1})}

\theoremstyle{plain}
\newtheorem{theo}{Theorem}[section]
\newtheorem{cor}[theo]{Corollary}
\newtheorem{prop}[theo]{Proposition}
\newtheorem{lem}[theo]{Lemma}
\newtheorem{defi}[theo]{Definition}
\newtheorem{theorem}[theo]{Theorem}
\newtheorem{definition}[theo]{Definition}
\newtheorem{hypothese}{Assumption}
\newtheorem{lemma}[theo]{Lemma}

\newtheorem{proposition}[theo]{Proposition}
\theoremstyle{remark}
\newtheorem{rem}[theo]{Remark}
\newtheorem{remark}[theo]{Remark}

\newcommand{\Ex}{\mathbb{N}}

\renewcommand{\Pr}{\mathbb P} 
\newcommand{\Er}{\mathbb E} 
\newcommand{\Esp}{\mathbf{E}}
\renewcommand{\phi}{\varphi}
\renewcommand{\epsilon}{\varepsilon}

\newcommand{\EsAr}{\mathbb{T}}

\newcommand{\cb}{{\mathcal B}}
\newcommand{\cc}{{\mathcal C}}

\newcommand{\cf}{{\mathcal F}}

\newcommand{\ch}{{\mathcal H}}

\newcommand{\cn}{{\mathcal N}}
\newcommand{\cm}{{\mathcal M}}
\newcommand{\cp}{{\mathcal P}}

\newcommand{\crr}{{\mathcal R}}

\newcommand{\ct}{{\mathcal T}}
\newcommand{\at}{{\mathfrak T}}

\newcommand{\cx}{{\mathcal X}}
\newcommand{\cy}{{\mathcal Y}}

\newcommand{\cz}{{\mathcal Z}}

\newcommand{\E}{{\mathbb E}}

\newcommand{\N}{{\mathbb N}}
\renewcommand{\P}{{\mathbb P}}

\newcommand{\R}{{\mathbb R}}

\newcommand{\T}{{\mathbb T}}

\newcommand{\K}{{\mathbb K}}

\newcommand{\bm}{{\mathbf m}}

\newcommand{\bN}{{\mathbf N}}
\newcommand{\bP}{{\mathbf P}}

\newcommand{\LL}{\mathbb L}

\newcommand{\ind}{{\bf 1}}

\newcommand{\dis}{{\rm dis}\;}

\newcommand{\norm}[1]{\mathop{\parallel\! #1 \! \parallel}\nolimits}
\newcommand{\val}[1]{\mathop{\left| #1 \right|}\nolimits}
\newcommand{\inv}[1]{\mathop{\frac{1}{ #1}}\nolimits}
\newcommand{\expp}[1]{\mathop {\mathrm{e}^{ #1}}}

\begin{document}
 
\title{Exit times for an increasing Lévy tree-valued process}
\date{\today}
\author{Romain Abraham} 

\address{
Romain Abraham,
MAPMO, CNRS UMR 7349,
F\'ed\'eration Denis Poisson FR 2964,
Universit\'e d'Orl\'eans,
B.P. 6759,
45067 Orl\'eans cedex 2
FRANCE.
}
 
\email{romain.abraham@univ-orleans.fr} 

\author{Jean-Fran\c{c}ois Delmas}
\address{
Jean-Fran\c cois Delmas,
Universit\'e Paris-Est, CERMICS, 6-8
av. Blaise Pascal,
 Champs-sur-Marne, 77455 Marne La Vall\'e, France.
\url{http://cermics.enpc.fr/~delmas/home.html}
}
\email{delmas@cermics.enpc.fr}

\author{Patrick Hoscheit}
\address{
Patrick Hoscheit, 
Universit\'e Paris-Est, CERMICS, 6-8
av. Blaise Pascal,
 Champs-sur-Marne, 77455 Marne La Vall\'e, France.
\url{http://cermics.enpc.fr/~hoscheip/home.html}
}
\email{hoscheip@cermics.enpc.fr}

\thanks{This work is partially supported the French ``Agence Nationale de
 la Recherche'', ANR-08-BLAN-0190.}



\begin{abstract} We give an explicit construction of the increasing 
 tree-valued process introduced by Abraham and Delmas using a random 
 point process of trees and a grafting procedure. This random point
 process will be used in companion papers to study record processes on
 Lévy trees. We use the Poissonian structure of the jumps of the
 increasing tree-valued process to describe its behavior at the first
 time the tree grows higher than a given height. We also give the
 joint distribution of this exit time and the ascension time which
 corresponds to the first infinite jump of the tree-valued process. 
\end{abstract}
\maketitle

\section{Introduction}

Lévy trees arise as a natural generalization to the continuum trees defined by
Aldous \cite{Aldous1991a}. They are located at the intersection of several
important fields: combinatorics of large discrete trees, Lévy processes and
branching processes. Consider a branching mechanism $\psi$, that is a function
of the form 
\begin{equation} 
\label{psi} 
\psi(\lambda) = \alpha \lambda + \beta
\lambda^2 + \int_{(0,+\infty)} (\expp{-\lambda x} -1 + \lambda
x\ind_{\{x<1\}}) \Pi(dx) 
\end{equation} 
with $\alpha\in \R$, $\beta\geq 0$, $\Pi$ a Lévy measure such that
$\int_{(0,+\infty )} 1\wedge x^2 \; \Pi(dx)<+\infty $. In the (sub)critical case
$\psi'(0)\ge 0$, Le Gall and Le Jan \cite{LeGall1998b} defined a
\emph{continuum} tree structure, which can be described by a tree $\ct$, for the
genealogy of a population whose size is given by a CSBP with branching mechanism
$\psi$. We shall consider the distribution $\P_r^\psi(d\ct)$ of this Lévy tree
when the CSBP starts at mass $r>0$, or its excursion measure $\N^\psi[d\ct]$,
when the CSBP is distributed under its canonical measure. The $\psi$-Lévy tree
possesses several striking features as pointed out in the works of Duquesne and
Le Gall \cite{Duquesne2002,Duquesne2005a}. For instance, the branching nodes can
only be of degree 3 (binary branching) if $\beta>0$ or of infinite degree (when
removing the branching point, the tree is separated in infinitely many connected
components) if $\Pi \ne 0$. Furthermore, there exists a mass measure $\bm^\ct$
on the leaves of $\ct$, whose total mass corresponds to the total population
size $\sigma=\bm^\ct(\ct)$ of the CSBP. We shall also consider the extinction
time of the CSBP which corresponds to the height $H_{max}(\ct)$ of the tree
$\ct$. The results can be extended to the super-critical case, using a Girsanov
transformation given by Abraham and Delmas \cite{Abraham2009b}.
\par In \cite{Abraham2009b}, a decreasing continuum tree-valued process is
defined using the so-called pruning procedure of Lévy trees introduced in
Abraham, Delmas and Voisin \cite{Abraham2010a}. By marking a $\psi$-Lévy tree
with two different kinds of marks (the first ones lying on the skeleton of the
tree, the other ones on the nodes of infinite degree), one can prune the tree by
throwing away all the points having a mark on their ancestral line connecting
them to the root. The main result of \cite{Abraham2010a} is that the remaining
tree is still a Lévy tree, with branching mechanism related to $\psi$. The idea
of \cite{Abraham2009b} is to consider a particular pruning with an intensity
depending on a parameter $\theta$, so that the corresponding branching mechanism
$\psi_\theta$ is $\psi$ shifted by $\theta$:
\[
\psi_\theta(\lambda)=\psi(\theta+\lambda)- \psi(\theta).
\]

Letting $\theta$ vary enables to define a decreasing tree-valued Markov process
$(\ct_\theta, \theta\in \Theta^\psi)$, with $\Theta^\psi\subset \R$ the set of
$\theta$ for which $\psi_\theta$ is well-defined, and such that $\ct_\theta$ is
distributed according to $\N^{\psi_\theta}$. If we write
$\sigma_\theta=\bm^{\ct_\theta}(\ct_\theta)$ for the total mass of $\ct_\theta$,
then the process $(\sigma_\theta, \theta \in \Theta^\psi)$ is a pure-jump
process. The case $\Pi=0$ was studied by Aldous and Pitman \cite{Aldous1998}.
The time-reversed tree-valued process is also a Markov process which defines a
growing tree process. Let us mention that the same kind of ideas have been used
by Aldous and Pitman \cite{Aldous1998a} and by Abraham, Delmas and He
\cite{Abraham2010} in the framework of Galton-Watson trees to define growing
discrete tree-valued Markov processes.
\par In the discrete framework of \cite{Abraham2010}, it is possible to define
the infinitesimal transition rates of the growing tree process. In
\cite{Evans2006}, Evans and Winter define another continuum tree-valued process
using a prune and re-graft procedure. This process is reversible with respect to
the law of Aldous's continuum random tree and its infinitesimal transitions are
described using the theory of Dirichlet forms.
\par In this paper, we describe the infinitesimal behavior of the growing
continuum tree-valued process, that is of $(\ct_\theta, \theta\in \Theta^\psi)$
seen backwards in time. The Special Markov Property in \cite{Abraham2010a}
describes only two-dimensional distributions and hence the transition
probabilities but, since the space of real trees is not locally compact, we
cannot use the theory of infinitesimal generators to describe its infinitesimal
transitions. Dirichlet forms cannot be used either since the process is not
symmetric (it is increasing). However, it is a pure-jump process and our first
main result shows that the infinitesimal transitions of the process can be
described using a random point process of trees which are grafted one by one on
the leaves of the growing tree. More precisely, let $\{\theta_j; j\in J\}$ be
the set of jumping times of the mass process $ (\sigma_\theta, \theta \in
\Theta^\psi)$. Then, informally, at time $\theta_j$, a tree $\ct^j$ distributed
according to $\bN^{\psi_{\theta_j}}[\ct\in \bullet]$, with:
\[
 \bN^{\psi_\theta}[\ct\in \bullet] = 2\beta \Ex^{\psi_\theta}[\ct\in
 \bullet] + 
\int_{(0,+\infty)} \Pi(dr) r \expp{-\theta r} \Pr_r^{\psi_\theta}(\ct
\in \bullet), 
\]
is grafted at $x_j$, a leaf of $\ct_{\theta_j}$ chosen at random (according to
the mass measure $\bm^{\ct_{\theta_j}}$). We also prove that the random point
measure 
\[
\cn=\sum_{j\in J}\delta_{(x_j,\ct^j,\theta_j)}
\]
has predictable compensator:
\[
\bm^{\ct_\theta}(dx) \bN^{\psi_\theta}[d\ct]\; \ind_{\{\theta\in
 \Theta^\psi\}}\; d\theta
\]
with respect to the backwards in time natural filtration of the process. See
Corollary \ref{cor:cn-T} for a precise statement. 
\par Notice that the precise statement relies on the introduction of the set of
locally compact weighted real trees endowed with a Gromov-Hausdorff-Prohorov
distance. Therefore, we will assume that Lévy trees are locally compact which
corresponds to the Grey condition: $\int^{+\infty}\frac{du}{\psi(u)}<\infty$. In
the (sub)critical case this implies that the corresponding height process of the
Lévy tree is continuous and that the tree is compact. However, the tree-valued
process is defined in \cite{Abraham2010a} without this assumption and we
conjecture that the jump representation of the tree-valued Markov process holds
without this assumption.
\par The representation using the random point measure allows to describe the
ascension time or explosion time (when it is defined):
\[
A = \inf \{ \theta \in \Theta^\psi,\ \sigma_\theta < \infty \}
\]
as $\inf\{\theta_j, \bm^{\ct^j}(\ct^j)<\infty\}$, the first time (backward!) a
tree with infinite mass is grafted. This representation is also used in Abraham
and Delmas \cite{Abraham2011,Abraham2012} respectively on the asymptotics of the
records on discrete subtrees of the continuum random tree and on the study of
the record process in general Lévy trees.
\par This structure, somewhat similar to the Poissonian structure of the jumps
of a Lévy process (although in our case the structure is neither homogeneous nor
independent), enables us to study the exit time of first passage of the growing
tree-valued process above a given height:
\[ 
A_h = \sup \{ \theta \in \Theta^\psi,\ H_{max}(\ct_\theta) > h\}. 
\]
We give the joint distribution of the ascension time and the exit time
$(A,A_h)$, see Proposition \ref{prop:Qh-A}. In particular, $A_h$ goes to $A$ as
$h$ goes to infinity: for $h$ very large, with high probability the process up
to $A$ will not have crossed height $h$, so that the first jump to cross height
$h$ will correspond to the grafting time of the first infinite tree, which
happens at the ascension time $A$.
\par We also give in Theorem \ref{ArbreAvantH} the joint distribution of
$(\ct_{A_h-}, \ct_{A_h})$ the tree just after and just before the jumping time
$A_h$. And we give a decomposition of $\ct_{A_h}$ along the ancestral branch of
the leaf on which the overshooting tree is grafted, which is similar to the
classical Bismut decomposition of Lévy trees. Conditionally on this ancestral
branch, the overshooting tree is then distributed as a regular Lévy tree,
conditioned on being high enough to perform the overshooting. This generalizes
results in \cite{Abraham2009b} about the ascension time of the tree-valued
process. Notice that this approach could easily be generalized to study spatial
exit times of growing families of super-Brownian motions.
\par All the results of this paper are stated in terms of real trees and not in
terms of the height process or the exploration process that encode the tree as
in \cite{Abraham2010a}. For this purpose, we define in Section \ref{sec:GPdist}
the state space of rooted real trees with a mass measure (called here weighted
trees or w-trees) endowed with the so-called Gromov-Hausdorff-Prohorov metric
defined in Abraham, Delmas and Hoscheit \cite{Abraham2012a} which is a slight
generalization of the Gromov-Hausdorff metric on the space of metric spaces, and
also a generalization of the Gromov-Prohorov topology of \cite{Greven2008b} on
the space of compact metric spaces endowed with a probability measure.
\par The paper is organized as follows. In Section \ref{sec:rappel}, we
introduce all the material for our study: the state space of weighted real trees
and the metric on it, see Section \ref{sec:GPdist} ; the definition of
sub(critical) Lévy trees via the height process ; the extension of the
definition to super-critical Lévy trees ; the pruning procedure of Lévy trees.
In Section \ref{sec:growing}, we recall the definition of the growing
tree-valued process by the pruning procedure as in \cite{Abraham2010a} in the
setting of real trees and give another construction using the grafting of trees
given by random point processes. We prove in Theorem \ref{theo:at=ct} that the
two definitions agree and then give in Corollary \ref{cor:cn-T} the random Point
measure description. Section \ref{sec:appli} is devoted to the application of
this construction on the distribution of the tree at the times it overshoots a
given height and just before, see Theorem \ref{ArbreAvantH}.

\section{The pruning of Lévy trees}
\label{sec:rappel}
\subsection{Real trees}

The first definitions of continuum random trees go back to Aldous
\cite{Aldous1991a}. Later, Evans, Pitman and Winter \cite{Evans2005} used the
framework of real trees, previously used in the context of geometric group
theory, to describe continuum trees. We refer to \cite{Evans,LeGall2006a} for a
general presentation of random real trees. Informally, real trees are metric
spaces without loops, locally isometric to the real line. 
\par More precisely, a metric space $(T,d)$ is a \emph{real tree} (or
$\R$\emph{-tree}) if the following properties are satisfied:
\begin{enumerate}
	\item For every $s,t\in T$, there is a unique isometric map $f_{s,t}$
from $[0,d(s,t)]$ to $T$ such that $f_{s,t}(0)=s$ and $f_{s,t}(d(s,t))=t$. 
	\item For every $s,t\in T$, if $q$ is a continuous injective map from
$[0,1]$ to $T$ such that $q(0)=s$ and $q(1)=t$, then
$q([0,1])=f_{s,t}([0,d(s,t)])$. 
\end{enumerate}
We say that a real tree is \emph{rooted} if there is a distinguished vertex
$\emptyset$, which will be called the \emph{root} of $T$. Such a real tree is
noted $(T,d,\emptyset)$. If $s,t\in T$, we will note $\llbracket s,t \rrbracket$
the range of the isometric map $f_{s,t}$ described above. We will also note
$\llbracket s,t \llbracket$ for the set $\llbracket s,t \rrbracket
\setminus\{t\}$. 
We give some vocabulary on real trees, which will be used constantly when
dealing with Lévy trees. Let $T$ be a real tree. If $x\in T$, we shall call
$\emph{degree}$ of $x$, and note by $n(x)$, the number of connected components
of the set $T\setminus\{x\}$. In a general tree, this number can be infinite,
and this will actually be the case with Lévy trees. The set of \emph{leaves} is
defined as:
\[
\mathrm{Lf}(T)=\{x\in T\backslash \{\emptyset\},\ n(x)=1\}. 
\]
If $n(x)\ge 3$, we say that $x$ is a \emph{branching point}. The set of
branching points will be noted $\mathrm{Br}(T)$. Among those, there is the set
of \emph{infinite} branching points, defined by
\[ 
\mathrm{Br}_\infty(T) = \{ x\in \mathrm{Br}(T),\ n(x) = \infty \} .
\]
Finally, the \emph{skeleton} of a real tree, noted $\mathrm{Sk}(T)$, is the set
of points in the tree that aren't leaves. It should be noted, following Evans,
Pitman and Winter \cite{Evans2005}, that the trace of the Borel $\sigma$-field
of $T$ on $\mathrm{Sk}(T)$ is generated by the sets $\llbracket s,s'
\rrbracket,\ s,s' \in \mathrm{Sk}(T)$. Hence, it is possible to define a
$\sigma$-finite Borel measure $l^{T}$ on $T$, such that 
\[ 
l^{T}(\mathrm{Lf}(T)) = 0
\quad\text{and}\quad
 l^{T}(\llbracket s,s'
\rrbracket)=d(s,s'). 
\]
This measure will be called \emph{length measure} on $T$. If $x,y$ are two
points in a rooted real tree $(T,d,\emptyset)$, then there is a unique point
$z\in T$, called the Most Recent Common Ancestor (MRCA) of $x$ and $y$ such that
$\llbracket \emptyset,x \rrbracket\cap \llbracket \emptyset,y \rrbracket =
\llbracket\emptyset,z \rrbracket$. This vocabulary is an illustration of the
genealogical vision of real trees, in which the root is seen as the ancestor of
the population represented by the tree. Similarly, if $x\in T$, we shall call
\emph{height} of $x$, and note by $H_x$ the distance $d(\emptyset,x)$ to the
root. The function $x\mapsto H_x$ is continuous on $T$, and we define the height
of $T$:
\[
H_{max}(T)=\sup_{x\in T} H_x.
\]

\subsection{Gromov-Prohorov metric}
\label{sec:GPdist}
\subsubsection{Rooted weighted metric spaces}
This Section is inspired by \cite{Duquesne2007}, but for the
fact that we include measures on the trees, in the spirit of
\cite{Miermont2007}. The detailed proofs of the results stated in this
Section are in \cite{Abraham2012a}. 
\par Let $(X,d^X)$ be a Polish metric space. For $A,B\in \cb(X)$, we set:
\[
 d_\text{H}^X(A,B)= \inf \{ \epsilon >0,\ A\subset B^\epsilon\
 \mathrm{and}\ B\subset 
A^\epsilon \}, 
\]
the Hausdorff distance between $A$ and $B$, where $A^\epsilon = \{ x\in X,
\inf_{y\in A} d^X(x,y) < \epsilon\}$ is the $\epsilon$-halo set of $A$. If $X$
is compact, then the space of compact subsets of $X$, endowed with the Hausdorff
distance, is compact, see theorem 7.3.8 in \cite{Burago2001}.
\par Recall that a Borel measure is locally finite if the measure of any bounded
Borel set is finite. We will use the notation $\cm_{f}(X)$ for the space of all
finite Borel measures on $X$. If $\mu,\nu \in \cm_f(X)$, we set:
\[
 d_\text{P}^X(\mu,\nu) = \inf \{ \epsilon >0,\ \mu(A)\le \nu(A^\epsilon) +
 \epsilon 
\text{ and } 
\nu(A)\le \mu(A^\epsilon)+\epsilon\ 
\text{ for all closed set } A \}, 
\]
the Prohorov distance between $\mu$ and $\nu$. It is well known that $(\cm_f(X),
d_\text{P}^X)$ is a Polish metric space, and that the topology generated by
$d_{P}^X$ is exactly the topology of weak convergence (convergence against
continuous bounded functionals).
\par If $\Phi:X\rightarrow X'$ is a Borel map between two Polish metric spaces
and if $\mu$ is a Borel measure on $X$, we will note $\Phi_*\mu$ the image
measure on $X'$ defined by $\Phi_*\mu(A)=\mu(\Phi^{-1}(A))$, for any Borel set
$A\subset X$. 

\begin{defi}
 \label{defi:rwms} 
$ $
\begin{itemize}
\item A \emph{rooted weighted metric space} $\cx = (X,d^X,
 \emptyset^X,\mu^X)$ is a metric space $(X , d^X)$ with a distinguished
 element $\emptyset^X\in X$ and a locally finite Borel measure $\mu^X$.
\item Two rooted weighted metric spaces $\cx=(X,d^X,\emptyset^X,\mu^X)$
 and $\cx'=(X',d^{X'},\emptyset^{X'},\mu^{X'})$ are said
 \emph{GHP-isometric} if there exists an isometric bijection $\Phi:X
 \rightarrow X'$ such that $\Phi(\emptyset^X)= \emptyset^{X'}$ and $\Phi_*
 \mu^X = \mu^{X'}$.
\end{itemize}
\end{defi}

Notice that if $(X, d^X)$ is compact, then a locally finite measure on $X$ is
finite and belongs to $\cm_f(X)$. We will now use a procedure due to Gromov
\cite{Gromov1999metric} to compare any two compact rooted weighted metric
spaces, even if they are not subspaces of the same Polish metric space.

\subsubsection{Gromov-Hausdorff-Prohorov distance for compact metric
 spaces} 
 Let $\cx=(X,d,\emptyset,\mu)$ and $\cx'=(X',d',\emptyset',\mu')$ be two compact
rooted weighted metric spaces, and define:
\begin{equation} \label{f:def}
 d_{\text{GHP}}^c(\cx,\cx') = \inf_{\Phi,\Phi',Z} \left(
 d_\text{H}^Z(\Phi(X),\Phi'(X')) + 
d^Z(\Phi(\emptyset),\Phi'(\emptyset')) + d_\text{P}^Z(\Phi_* \mu,\Phi_*'
\mu') \right), 
\end{equation}
where the infimum is taken over all isometric embeddings $\Phi:X\hookrightarrow
Z$ and $\Phi':X'\hookrightarrow Z$ into some common Polish metric space
$(Z,d^Z)$. 
\par Note that equation \reff{f:def} does not actually define a distance
function, as $d_{\text{GHP}}^c(\cx,\cx')=0$ if $\cx$ and $\cx'$ are
GHP-isometric. Therefore, we shall consider $\K$, the set of GHP-isometry
classes of compact rooted weighted metric space and identify a compact rooted
weighted metric space with its class in $\K$. Then the function
$d_{\text{GHP}}^c$ is finite on $\K^2$. 

\begin{theo}
 \label{theo:dcGHP}
The function $d_{\text{GHP}}^c$ defines a metric on $\K$ and the space
$(\K, d_{\text{GHP}}^c)$ is a Polish metric space. 
\end{theo}

We shall call $d_{\text{GHP}}^c$ the \emph{Gromov-Hausdorff-Prohorov} metric.
This extends the Gromov-Hausdorff metric on compact metric spaces, see
\cite{Burago2001} section 7, as well as the Gromov-Hausdorff-Prohorov metric on
compact metric spaces endowed with a probability measure, see
\cite{Miermont2007}. See also \cite{Greven2008b} for an other approach on metric
spaces endowed with a probability measure. 

\subsubsection{Gromov-Hausdorff-Prohorov distance}

However, the definition of Gromov-Hausdorff-Prohorov distance on compact metric
space is not yet general enough, as we want to deal with unbounded trees with
$\sigma$-finite measures. To consider such an extension, we shall consider
complete and locally compact length spaces.
\par We recall that a metric space $(X,d)$ is a \emph{length space} if for every
$x,y\in X$, we have:
\[
 d(x,y) = \inf L(\gamma) , 
\]
where the infimum is taken over all rectifiable curves $\gamma:[0,1]\rightarrow
X$ such that $\gamma(0)=x$ and $\gamma(1)=y$, and where $L(\gamma)$ is the
length of the rectifiable curve $\gamma$. 

\begin{defi}
 \label{defi:L} Let $\LL$ be the set of GHP-isometry classes of rooted weighted
complete and locally compact length spaces and identify a rooted weighted
complete and locally compact length spaces with its class in $\LL$.
\end{defi}

If $\cx=(X,d,\emptyset,\mu)\in \LL$, then for $r\geq 0$ we will consider its
restriction to the ball of radius $r$ centered at $\emptyset$,
$\cx^{(r)}=(X^{(r)}, d^{(r)}, \emptyset, \mu^{(r)})$, where
\[
X^{(r)}=\{x\in X; d(\emptyset,x)\leq r\},
\]
the metric $d^{(r)}$ is the restriction of $d$ to $X^{(r)}$, and the measure
$\mu^{(r)}(dx)=\ind_{X^{(r)}} (x)\; \mu(dx)$ is the restriction of $\mu$ to
$X^{(r)}$. Recall the Hopf-Rinow theorem implies that if $(X, d)$ is a complete
and locally compact length space, then every closed bounded subset of $X$ is
compact. In particular if $\cx$ belongs to $ \LL$ , then $\cx^{(r)}$ belongs to
$\K$ for all $r\geq 0$.
\par We state a regularity Lemma of $d^c_{\text{GHP}}$ with respect to the
restriction operation. 

\begin{lem}
 \label{lem:reg-d-GHP}
Let $\cx$ and $\cy$ belong to $\LL$. Then the function defined on $\R_+$:
\[
r\mapsto d^c_{\text{GHP}}\left(\cx^{(r)},\cy^{(r)}\right)
\]
is càdlàg.
\end{lem}
This implies that the following function is well defined on $\LL^2$:
\[
 d_{\text{GHP}}(\cx,\cy) = \int_0^\infty \expp{-r} \left(1 \wedge
d^c_{\text{GHP}}\left(\cx^{(r)},\cy^{(r)}\right)
\right) \ dr.
\]


\begin{theo}
 \label{theo:LL} The function $d_{\text{GHP}}$ defines a metric on $\LL$ and the
space $(\LL, d_{\text{GHP}})$ is a Polish metric space.
\end{theo}

The next result, implies that $d_{\text{GHP}}^c$ and $d_{\text{GHP}}$ defines
the same topology on $\K\cap \LL$. 
\begin{theo}
 \label{theo:K-LL} Let $(\cx_n, n\in \N)$ and $\cx$ be elements of $\K\cap
 \LL$. Then the sequence $(\cx_n, n\in \N)$ converges to $\cx$ in
 $(\K,d_{\text{GHP}}^c)$ if and only if it converges to $\cx$ in
 $(\LL,d_{\text{GHP}})$. 
\end{theo}

\subsubsection{The space of w-trees}

Note that real trees are always length spaces and that complete real trees are
the only complete connected spaces that satisfy the so-called \emph{four-point
condition}:
\begin{equation}
 \forall x_1,x_2,x_3,x_4 \in X,\ d(x_1,x_2)+d(x_3,x_4) \le
(d(x_1,x_3)+d(x_2,x_4) ) \vee (d(x_1,x_4)+d(x_2,x_3)). 
\end{equation}

\begin{definition} 
We denote by $\T$ be the set of (GHP-isometry classes of) complete locally
compact rooted real trees endowed with a locally finite Borel measure,
in short \emph{w-trees}. 
\end{definition}

We deduce the following Corollary from Theorem \ref{theo:LL} and the
four-point condition characterization of real trees. 

\begin{cor}
 \label{cor:T}
The set $\T$ is a closed subset of $\LL$ and 
$(\T,d_{\text{GHP}})$ is a Polish metric space. 
\end{cor}

\subsection{Height erasing}

We define the restriction operators on the space of w-trees. Let $a\geq 0$.
If $(T,d,\emptyset,\bm)$ is a w-tree, let
\begin{equation}
\label{eq:def-pi}
 \pi_a(T) = \{ x \in T,\
d(\emptyset,x) \le a \}
 \end{equation}
and $(\pi_a(T), d^{\pi_a(T)}, \emptyset,\bm^{\pi_a(T)})$ be the w-tree
constituted of the points of $T$ having height lower than $a$, where
$d^{\pi_a(T)}$ and $\bm^{\pi_a(T)}$ are the restrictions of $d$ and $\bm$ to
$\pi_a(T)$. When there is no confusion, we will also write $\pi_a(T)$ for
$(\pi_a(T), d^{\pi_a(T)}, \emptyset,\bm^{\pi_a(T)})$. We will also write
$T(a)=\{ x\in T,\ d(\emptyset,x)=a\}$ for the level set at height $a$. We say a
w-tree $T$ is bounded if $\pi_a(T)=T$ for some finite $a$. Notice that a tree
$T$ is bounded if and only if $H_{max}(T)$ is finite.

\subsection{Grafting procedure}

We will define in this section a procedure by which we add (graft) w-trees on an
existing w-tree. More precisely, let $T \in \T$ and let $((T_i,x_i),i\in I)$ be
a finite or countable family of elements of $\T\times T$. We define the real
tree obtained by grafting the trees $T_i$ on $T$ at point $x_i$. We set
$\tilde{T} = T \sqcup \left( \bigsqcup_{i\in I} T_i\backslash\{\emptyset^{T_i}\}
\right) $ where the symbol $\sqcup$ means that we choose for the sets $T$ and
$(T_i)_{i\in I}$ representatives of isometry classes in $\T$ which are disjoint
subsets of some common set and that we perform the disjoint union of all these
sets. We set $\emptyset^{\tilde T}=\emptyset^T$. The set $\tilde T$ is endowed
with the following metric $d^{\tilde T}$: if $s,t\in \tilde T$,
\begin{equation*} 
d^{\tilde T} (s,t) = 
\begin{cases} 
d^T(s,t)\ & \text{if}\ s,t\in T, \\
d^T(s,x_i)+d^{T_i}(\emptyset^{T_i},t)\ & \text{if}\
s\in T,\ t\in T_i\backslash\{\emptyset^{T_i}\} , \\ 
d^{T_i}(s,t)\ & \text{if}\ s,t\in T_i\backslash\{\emptyset^{T_i}\} ,\\
d^T(x_i,x_j)+d^{T_j}(\emptyset^{T_j},s)+d^{T_i}
(\emptyset^{T_i},t)\ 
& \text{if}\ i\neq j \ \text{and}\ s\in T_j\backslash\{\emptyset^{T_j}\} ,\ t\in
T_i\backslash\{\emptyset^{T_i}\} .
\end{cases} 
\end{equation*}
We define the mass measure on $\tilde T$ by:
\[
\mathbf{m}^{\tilde T}=\mathbf{m}^T+\sum_{i\in I}\ind_{
 T_i\backslash\{\emptyset^{T_i}\}} \mathbf{m}^{T_i}+
\mathbf{m}^{T_i}(\{\emptyset^{T_i}\}) \delta_{x_i},
\]
where $\delta_x$ is the Dirac mass at point $x$. It is clear that the metric
space $(\tilde{T},d^{\tilde T},\emptyset^{\tilde T})$ is still a rooted complete
real tree. However, it is not always true that $\tilde{T}$ remains locally
compact (it still remains a length space anyway), or, for that matter, that
$\mathbf{m}^{\tilde T}$ defines a
locally finite measure (on $\tilde T$). So, we will have to check that $(\tilde
T,d^{\tilde T},\emptyset^{\tilde T}, \mathbf{m}^{\tilde T} )$ is a w-tree in the
particular cases we will consider.

\par We will use the following notation:
\begin{equation} 
(\tilde T,d^{\tilde T},\emptyset^{\tilde T},
\mathbf{m}^{\tilde T} ) = T \circledast_{i\in I}(T_i,x_i) 
\label{def:gref}
\end{equation}
and write $\tilde T$ instead of $(\tilde T,d^{\tilde T},\emptyset^{\tilde T},
\mathbf{m}^{\tilde T} ) $ when there is no confusion.

\subsection{Real trees coded by functions}
\label{sec:code}
Lévy trees are natural generalizations of Aldous's Brownian tree, where the
underlying process coding for the tree (reflected Brownian motion in Aldous's
case) is replaced by a certain functional of a Lévy process, the \emph{height
process}. Le Gall and Le Jan \cite{LeGall1998b} and Duquesne and Le Gall
\cite{Duquesne2005a} showed how to generate random real trees using the
excursions of a Lévy process. We shall briefly recall this construction, in
order to introduce the pruning procedure on Lévy trees. Let us first work in a
deterministic setting. 
\par Let $f$ be a continuous non-negative function defined on $[0,+\infty)$,
such that $f(0)=0$, with compact support. We set:
\[
\sigma^f=\sup\{t; f(t)>0\},
\]
with the convention $\sup\emptyset=0$. Let $d^f$ be the non-negative
function defined by:
\[ 
d^f(s,t) = f(s) + f(t) - 2 \inf_{u\in [s\wedge t , s\vee t]} f(u). 
\]
It can be easily checked that $d^f$ is a semi-metric on $[0,\sigma^f]$. One can
define the equivalence relation associated to $d^f$ by $s\sim t$ if and only if
$d^f(s,t)=0$. Moreover, when we consider the quotient space
\begin{equation*}
 T^f=[0,\sigma^f]_{/ \sim}
\end{equation*}
and, noting again $d^f$ the induced metric on $T^f$ and rooting $T^f$ at
$\emptyset^f$, the equivalence class of 0, it can be checked that the
space $(T^f,d^f,\emptyset^f)$ is a compact rooted real tree. We denote
by $p^f$ the canonical projection from $[0,\sigma^f]$ onto $T^f$, which
is extended by $p^f(t)=\emptyset^f$ for $t\geq \sigma^f$. Notice that
$p^f$ is continuous. We define $\mathbf{m}^{f}$, the mass measure on
$T^f$ as the image measure on $T^f$ of the Lebesgue measure on $[0,
\sigma^f]$ by $p^f$. We consider the (compact) w-tree $(T^f, d^f,
\emptyset^f, \mathbf{m}^f)$, which we shall denote $T^f$.

\par It should be noted that, if $x\in T^f$ is an equivalence class, the common
value of $f$ on all the points in this equivalence class is exactly
$d^f(\emptyset,x)=H_x$. Notice that, in this setting, $H_{max}(T^f)=\|
f\|_\infty$ where $\| f\|_\infty$ stands for the uniform norm of $f$. 
\par We have the following elementary result (see Lemma 2.3 of
\cite{Duquesne2005a} when dealing with the Gromov-Hausdorff metric instead of
the Gromov-Hausdorff-Prohorov metric). 
\begin{prop}
 \label{prop:cont-TH}
 Let $f,g$ be two compactly supported, non-negative continuous functions with
$f(0)=g(0)=0$. Then:
\begin{equation}
 d_{GHP}^c(T^f,T^g) \le 6 \| f-g \|_\infty + | \sigma^f-\sigma^g |.
\end{equation}
\end{prop}

\begin{proof} The Gromov-Hausdorff distance can be evaluated using
 correspondences, see \cite{Burago2001}, section 7.3. A correspondence
 between two metric spaces $(E_1, d_1)$ and $(E_2, d_2)$ is a subset
 $\crr$ of $E_1\times E_2$ such that for $\delta\in \{1,2\}$ the
 projection of $\crr$ on $E_\delta$ is onto: $\{x_\delta; (x_1, x_2)\in
 \crr\}=E_\delta$. The distortion of $\crr$ is defined by:
\[
\dis (\crr)= \sup\left\{|d_1(x_1, y_1) - d_2(x_2,y_2)|;\; (x_1, y_1)\in
 \crr, (x_2, y_2)\in \crr\right\}.
\]
Let $Z=E_1 \sqcup E_2$ by the disjoint union of $E_1$ and $E_2$ and
consider the function $d^Z$ defined on $Z^2$ by: $d^Z= d_\delta$ on $E_\delta^2$
for $\delta\in
\{1,2\}$ and for $x_1\in E_1$, $x_2\in E_2$:
\[
d^Z(x_1, x_2)=\inf\left\{ d_1(x_1,y_1)+\inv{2}\dis (\crr)+ d_2(y_2, x_2);
 \; (y_1, y_2)\in \crr\right\}.
\]
Then if $\dis (\crr)>0$, the function $d^Z$ is a metric on $Z$. And we have:
\[
d^Z_H(E_1, E_2)\leq \inv{2} \dis (\crr).
\]

Let $f,g$ be two compactly supported, non-negative continuous functions
with $f(0)=g(0)=0$. Following \cite{Duquesne2005a}, we consider the
following correspondence between $\ct^f$ and $\ct^g$:
\[
\crr=\left\{(x^f,x^g); \; x^f=p^f(t) \text{ and } x^g=p^g(t) \text{ for
 some } t\geq 0\right\},
\]
and we have $\dis(\crr)\leq 4 \norm{f-g}_\infty $ according to the proof
of Lemma 2.3 in \cite{Duquesne2005a}. Notice $(\emptyset^f,
\emptyset^g)\in \crr$. Thus, with the notation above 
and $E_1=T^f$, $E_2=T^g$, we get:
\[
d_H^Z(T^f, T^g)\leq 2 \norm{f-g}_\infty 
\quad\text{and}\quad
d^Z(\emptyset^f, \emptyset^g)\leq 2 \norm{f-g}_\infty .
\]

Then, we consider the Prohorov distance between $\mathbf{m}^f$
and $\mathbf{m}^g$. Let $A^f$ be a Borel set of $T^f$. We set $I=\{t\in
[0, \sigma^f]; \; p^f(t)\in A\}$. By definition of $\bm^f$, we have
$\bm^f(A^f)=\text{Leb} (I)$. We set $A^g=p^g(I \cap [0, \sigma^g] )$ so
that $\bm^g(A^g) = \text{Leb} (I \cap [0, \sigma^g] )\geq \text{Leb}(I)
- |\sigma^f -\sigma^g|$. By construction, we also have that for any
$x^g\in A^g$, there exists $t\in I$ such that $p^g(t)=x^g$ and
$d^Z(x^g,x^f)=\inv{2}\dis(\crr)$, with $x^f=p^f(t)\in A^f$. This implies
that $A^g \subset
(A^f)^r$ for any $r>\inv{2}\dis(\crr)$. We deduce that:
\[
\bm^f(A^f)\leq \bm^g(A^g)+ |\sigma^f -\sigma^g|
\leq \bm^g\left((A^f)^r\right)+ |\sigma^f -\sigma^g|.
\]
The same is true with $f$ and $g$ replaced by $g$ and $f$. We deduce
that:
\[
d^Z_P(\bm^f, \bm^g)\leq \inv{2}\dis(\crr)+|\sigma^f -\sigma^g|
\leq 2 \norm{f-g}_\infty +|\sigma^f -\sigma^g|.
\]
We get:
\[
d_H^Z(T^f, T^g)+d^Z(\emptyset^f, \emptyset^g)+ d^Z_P(\bm^f, \bm^g)\leq 6
\| f-g \|_\infty + | \sigma^f-\sigma^g |. 
\]
This gives the result. 
\end{proof}

\begin{rem}
 \label{rem:Assumption2}
 We could define the correspondence for more general functions $f$: lower
semi-continuous functions that satisfy the intermediate values property (see
\cite{Duquesne2002}). In that case, the associated real tree is not even locally
compact (hence not necessarily proper). But the measurability of the mapping
$f\mapsto T^f$ is not clear in this general setting, that is why we only
consider continuous function $f$ here and thus will assume the Grey condition
(see next Section) for Lévy trees.
\end{rem}
\subsection{Branching mechanisms}

Let $\Pi$ be a $\sigma$-finite measure on $(0,+\infty )$ such that we have $\int
(1 \wedge x^2) \Pi(dx) < \infty$. We set:
\begin{equation}
 \label{eq:pi-q}
\Pi_{\theta} (dr)=\expp{-\theta r}\ \Pi(dr).
\end{equation}
Let $\Theta'$ be the set of $\theta\in \R$ such that $ \int_{(1,+\infty)}
\Pi_\theta (dr) < +\infty$. If $\Pi=0$, then $\Theta'=\R$. We also set
$\theta_\infty = \inf \Theta'$. It is obvious that $[0,+\infty )\subset
\Theta'$, $\theta_\infty \leq 0$ and either $\Theta'=[\theta_\infty , +\infty )$
or $\Theta'=(\theta_\infty , +\infty )$. 
\par Let $\alpha\in \R$ and $\beta \ge 0$. We consider the branching mechanism
$\psi$ associated with $(\alpha,\beta,\Pi)$:
\begin{equation}
\label{eq:def-psi}
 \psi(\lambda) = \alpha \lambda + \beta \lambda^2 + \int_{(0,+\infty)}
(\expp{-\lambda r} -1 + \lambda r
\ind_{\{r<1\}}) \Pi(dr) , \quad \lambda\in \Theta'.
\end{equation}
Notice that the function $\psi$ is smooth and convex over $(\theta_\infty ,
+\infty )$. We say that $\psi$ is conservative if for all $\varepsilon>0$:
\[
 \int_{(0,\varepsilon]} \frac{du}{|\psi(u)|} = +\infty. 
\] 
A sufficient condition for $\psi$ to be conservative is to have
$\psi'(0+)>-\infty$. This last condition is actually equivalent to
$\int_{(1,\infty)} r \Pi(dr) < \infty$. We will always make the following
assumption. 
\begin{hypothese} 
The function $\psi$ is conservative and we have $\beta>0$ or $\int_{(0,1)}
 \ell \Pi(d\ell)=+\infty $.
\end{hypothese}

The branching mechanism is said to be sub-critical (resp. critical,
super-critical) if $\psi'(0+) >0$ (resp. $\psi'(0+)=0$, $\psi'(0+) <0$). We say
that $\psi$ is (sub)critical if it is critical or sub-critical. 
\par We introduce the following branching mechanisms $\psi_\theta$ for
$\theta\in \Theta'$:
\begin{equation}
 \label{eq:def-psi-q}
\psi_\theta(\lambda)=\psi(\lambda+\theta)-\psi(\theta), \quad
\lambda+\theta\in \Theta'. 
\end{equation}
Let $\Theta^\psi$ be the set of $\theta\in \Theta'$ such that $\psi_\theta$ is
conservative. Obviously, we have: 
\[
[0,+\infty )\subset \Theta^\psi \subset \Theta' \subset \Theta^\psi
\cup \{\theta_\infty\}. 
\]
If $\theta\in \Theta^\psi$, we set:
\begin{equation}
 \label{eq:def-barq}
\bar \theta=\max\{ q\in \Theta^\psi; \psi(q)=\psi(\theta)\}.
 \end{equation} 
 We can give an alternative definition of $\bar \theta$ if Assumption 1
 holds. Let $\theta^*$ be the unique positive root of $\psi'$ if it
 exists. Notice that $\theta^*=0$ if $\psi$ is critical and that
 $\theta^*$ exists and is positive if $\psi$ is super-critical. If
 $\theta^*$ exists, then the branching mechanism $\psi_{\theta^*}$ is
 critical. We set $\Theta_*^\psi$ for $[\theta^*,+\infty )$ if
 $\theta^*$ exists and $\Theta_*^\psi=\Theta^\psi$ otherwise. The
 function $\psi$ is a one-to-one mapping from $\Theta_*^\psi$ onto
 $\psi(\Theta^\psi_*)$. We write $\psi^{-1}$ for the inverse of the
 previous mapping. The set $\{ q\in \Theta^\psi;
 \psi(q)=\psi(\theta)\}$ has at most two elements and we have:
\[
\bar\theta=\psi^{-1} \circ \psi(\theta).
\] 
In particular, if $\psi_\theta$ is (sub)critical we have $\bar \theta=
\theta$ and if $\psi_\theta$ is super-critical then we have
$\theta<\theta^*<\bar \theta$.

We will later on consider the following assumption. 
\begin{hypothese} (Grey condition) The branching mechanism is such that:
 \[ 
\int^{+\infty} \frac{du}{\psi(u)} < \infty. 
\]
\end{hypothese}

Let us remark that Assumption 2 implies that $\beta>0$ or $\int_{(0,1)}
 r \Pi(dr)=+\infty $.

\subsection{Connections with branching processes}

Let $\psi$ be a branching mechanism satisfying Assumption 1. A continuous state
branching process (CSBP) with branching mechanism $\psi$ and initial mass
$x>0$ is the càdlàg $\R_+$-valued Markov process $(Z_a,\ a\ge 0)$ whose
distribution is characterized by $Z_0=x$ and:
\begin{equation*}
 \Esp [ \exp(-\lambda Z_{a+a'}) | Z_a ] = \exp(-Z_a u(a',\lambda) ),
 \quad \lambda\geq 0, 
\end{equation*}
 where $(u(a,\lambda),a\ge0, \lambda>0)$ is the unique non-negative solution to
the integral equation:
\begin{equation}
\label{eq:def-u}
 \int_{u(a,\lambda)}^\lambda \frac{dr}{\psi(r)} = a\ ;\ \
u(0,\lambda)=\lambda.
\end{equation}
The distribution of the CSBP started at mass $x$ will be noted $\bP^\psi_x$. For
a detailed presentation of CSBPs, we refer to the monographs
\cite{Kyprianou2006},\cite{Lambert2008a} or \cite{Li2011}.
\par In this context, the conservative assumption is equivalent to the CSBP not
blowing up in finite time, and Assumption 2 is equivalent to the strong
extinction time, $\inf\{a; Z_a=0\}$, being a.s. finite. If Assumption 2 holds,
then for all $h>0$, $\bP_x^\psi(Z_{h} >0) = \exp(-x b(h))$, where
$b(h)=\lim_{\lambda \rightarrow +\infty } u(h,\lambda)$. In particular $b(h)$ is
such that
\begin{equation}
\label{bh} 
\int_{b(h)}^\infty \frac{dr}{\psi(r)} = h. 
\end{equation}

Let us now describe a Girsanov transform for CSBPs introduced in
\cite{Abraham2009b} related to the shift of the branching mechanism $\psi$
defined by \reff{eq:def-psi-q}. Recall notation $\Theta^\psi$ and $\theta_\infty
$ from the previous Section. For $\theta\in \Theta^\psi$, we consider the
process $M^{\psi,\theta}=(M^{\psi,\theta}_a, a\geq 0)$ defined by: 
 \begin{equation} 
\label{eq:mart-Ma}
M_a^{\psi,\theta} = \exp\Big(\theta x-\theta
Z_a-\psi(\theta)
\int_0^a Z_s ds\Big) .
\end{equation} 

\begin{theorem}[Girsanov transformation for CSBPs, \cite{Abraham2009b}]
\label{GirCSBP} 
Let $\psi$ be a branching mechanism satisfying Assumption 1. Let $(Z_a,a\ge 0)$
be a CSBP with branching mechanism $\psi$ and let $\cf = (\cf_a,a\ge 0)$ be its
natural filtration. Let $\theta\in \Theta^\psi$ such that either $\theta\geq 0$
or $\theta<0$ and $\int_{(1,+\infty )} r \Pi_\theta(dr)<+\infty $. Then we have
the following:
\begin{enumerate}
 \item The process $M^{\psi,\theta}$ is a $\cf$-martingale under
$\bP_x^\psi$.
 \item Let $a,x\ge 0$. On $\cf_a$, the probability measure
$\bP_x^{\psi_\theta}$ is absolutely continuous w.r.t. $\bP_x^\psi$, and 
\[
\frac{{d\bP_x^{\psi_\theta}}_{|\cf_a}}{{d\bP_x^{\psi}}_{|\cf_a}}
= M_a^{\psi,\theta}. 
\]
\end{enumerate}
\end{theorem}

\subsection{The height process}

Let $(X_t, t\ge 0)$ be a Lévy process with Laplace exponent $\psi$ satisfying
Assumption 1. This assumption implies that a.s. the paths of $X$ have infinite
total variation over any non-trivial interval. The distribution of the Lévy
process will be noted $\Pr^\psi(dX)$. It is a probability measure on the
Skorokhod space of real-valued càdlàg processes. For the remainder of this
section, \textbf{we will assume that $\psi$ is (sub)critical}. 

For $t\ge 0$, let us write $\hat{X}^{(t)}$ for the time-returned process:
\begin{equation*}
 \forall\ 0\le s< t,\ \hat{X}^{(t)}_s= X_t-X_{(t-s)_-}
\end{equation*}
and $\hat{X}^{(t)}_t=X_t$. Then $(\hat{X}^{(t)}_s,0\le s\le t)$ has same
distribution as the process $(X_s,0\le s\le t)$. We will also write
$\hat{S}^{(t)}_s = \sup_{[0,s]} \hat{X}^{(t)}_r$ for the supremum process of
$\hat{X}^{(t)}$. 
 
\begin{proposition}[The height process, \cite{Duquesne2002}] Let $\psi$ be a
(sub)critical branching mechanism satisfying Assumption 1. There exists
a lower semi-continuous process $H=(H_t,t\ge 0)$ taking values in $[0,+\infty
]$, with the intermediate values property, which is a local time at 0, at time
$t$, of the process $\hat{X}^{(t)}-\hat{S}^{(t)}$, such that the following
convergence holds in probability:
\begin{equation*}
H_t = \lim_{\epsilon \downarrow 0} \frac{1}{\epsilon} \int_0^t
\ind_{\{I_s^t \le X_s \le I_s^t +\epsilon\}} ds 
\end{equation*}
where $I_s^t=\inf_{s\le r \le t} X_r$. Furthermore, if Assumption 2
holds, then the process $H$ admits a continuous modification. 
\end{proposition}

From now on, we always assume that Assumptions 1 and 2 hold, and we
always work with this continuous version of $H$.
The process $H$ is called the height process. 
\par For $x>0$, we consider the stopping time $ \tau_x = \inf \Big\{ t\ge 0,\
I_t \le -x \Big\} $, where $I_t=I^t_0$ is the infimum process of $X$. We denote
by $\Pr^\psi_x(dH) $ the distribution of the stopped height process $(H_{t\wedge
\tau_x}, t\ge 0)$ under $\Pr^\psi$, defined on the space
$\mathcal{C}_+([0,+\infty))$ of non-negative continuous functions on
$[0,+\infty)$. The (sub)criticality of the branching mechanism entails
$\tau_x<\infty$ $\Pr^\psi$-a.s., so that under $\Pr^\psi_x(dH)$, a.s. the height
process has compact support.
\subsection{The excursion measure}

The height process is not a Markov process, but it has the same zero sets as
$X-I$ (see \cite{Duquesne2002}, Paragraph 1.3.1), so we can develop an excursion
theory based on the latter. By standard fluctuation theory, it is easy to see
that 0 is a regular point for $X-I$ and that $-I$ is a local time of $X-I$ at 0.
We denote by $\Ex^\psi$ the associated excursion measure. As such, $\Ex^\psi$ is
a $\sigma$-finite measure. Under $\Pr^\psi_x$ or $\Ex^\psi$, we set:
\[
\sigma(H) = \int_0^\infty \ind_{\{H_t \neq 0\}} dt.
\]
When there is no risk of confusion, we will write $\sigma$ for $\sigma(H)$.
Notice that, under $\Pr_x^\psi$, $\sigma=\tau_x$ and that under $\Ex^\psi$,
$\sigma$ represents the lifetime of the excursion. Abusing notations, we will
write $\P_x^\psi(dH)$ and $\N^\psi[dH]$ for the distribution of $H$ under
$\P_x^\psi$ or $\N^\psi$. Let us also recall the Poissonian decomposition of the
measure $\Pr^\psi_x$. Under $\Pr^\psi_x$, let $(a_j,b_j)_{j\in J}$ be the
excursion intervals of $X-I$ away from 0. Those are also the excursion intervals
of the height process away from $0$. For $j\in J$, we shall denote by
$H^{(j)}:[0,\infty)\rightarrow \R_+$ the corresponding excursion, that is
\[ 
\ H^{(j)}_t = H_{(a_j+t)\wedge b_j}, \quad t\geq 0.
\]
\begin{proposition}[\cite{Duquesne2005a}] 
\label{Buisson} 
Let $\psi$ be a (sub)critical branching mechanism satisfying Assumption 1. Under
$\Pr^\psi_x$, the random point measure $\cn = \sum_{j\in J}
\delta_{H^{(j)}}(dH)$ is a Poisson point measure with intensity
$x\Ex^\psi[dH]$.
\end{proposition}

\subsection{Local times of the height process}

\begin{prop}[\cite{Duquesne2002}, Formula (36)]
\label{prop:LT} 
Let $\psi$ be a (sub)critical branching mechanism satisfying Assumption 1. Under
$\N^\psi$, there exists a
jointly measurable process $(L^a_s, a\geq 0, s\geq 0)$ which is
continuous and non-decreasing in the variable $s$ such that,
$$\forall s\ge 0,\ L_s^0=0$$
and for every $t\geq 0$, for every $\delta>0$ and every $a>0$
$$\lim_{\varepsilon\rightarrow 0}\N^\psi \left[\ind_{\{\sup
 H>\delta\}}\sup_{0\le s\leq t\wedge\sigma}\left|\varepsilon^{-1}
\int_0^s \ind_{\{ a<H_r\leq a+\varepsilon\}}\;dr - L^a_s\right|\right]=0.$$ 
\end{prop}

Moreover, by Lemma 3.3. of \cite{Duquesne2005a}, the process
$(L_\sigma^a,\ a\ge 0)$ has a c\`adl\`ag modification under $\N^\psi$
with no fixed discontinuities.

\subsection{(Sub)critical Lévy trees}

Let $\psi$ be a (sub)critical branching mechanism satisfying Assumptions
1 and 2. Let $H$ be the height process defined under $\P^\psi_x$ or
$\N^\psi$. We consider the so-called Lévy tree $\ct^H$ which is the
random w-tree coded by the function $H$, see Section
\ref{sec:code}. Notice that we are indeed within the framework of proper
real trees, since Assumption 2 entails compactness of $\ct^H$. The
measurability of the random variable $\ct^H$ taking values in $\T$
follows from Proposition \ref{prop:cont-TH} and Theorem
\ref{theo:K-LL}. When there is no
confusion, we shall write $\ct$ for $\ct^H$. Abusing notations, we will
write $\P_x^\psi(d\ct)$ and $\N^\psi[d\ct]$ for the distribution on $\T$
of $\ct=\ct^H$ under $\P_x^\psi(dH)$ or $\N^\psi[dH]$. By construction,
under $\P^\psi_x$ or under $\N^\psi$, we have that the total mass of the
mass measure on $\ct$ is given by:
\begin{equation} 
\label{eq:s=la}
\mathbf{m}^{\ct}(\ct) = \sigma.
\end{equation}

Proposition \ref{Buisson} enables us to view the measure $\Ex^\psi[d\ct]$ as
describing a single Lévy tree. Thus, we will mostly work under this excursion
measure, which is the distribution of the (isometry class of the) w-tree $\ct$
described by the height process under $\Ex^\psi$. In order to state the
branching property of a L\'evy tree, we must first define a local time at level
$a$ on the tree. Let $(\ct^{i,\circ},i\in I)$ be the trees that were cut off by
cutting at level $a$, namely the connected components of the set
$\ct\setminus\pi_a(\ct)$. If $i\in I$, then all the points in $\ct^{i,\circ}$
have the same MRCA $x_i$ in $\ct$ which is precisely the point where the tree
was cut off. We consider the compact tree $\ct^{i}=\ct^{i,\circ} \cup \{x_i\}$
with the root $x_i$, the metric $d^{\ct^i}$, which is the metric $d^\ct$
restricted to $\ct^i$, and the mass measure $\mathbf{m}^{\ct^i}$, which is the
mass measure $\mathbf{m}^\ct$ restricted to $\ct^i$. Then $(\ct^i, d^{\ct^i},
x_i, \mathbf{m}^{\ct^i})$ is a w-tree. Let
\begin{equation} 
\label{eq:branchement}
\cn_a^{\ct}(dx,d\ct') = \sum_{i\in I}
\delta_{(x_i,\ct^i)}(dx,d\ct') 
\end{equation} 
be the point measure on $\ct(a)\times \EsAr$ taking account of the cutting
points as well as the trees cut away. The following theorem gives the structure
of the decomposition we just described. From excursion theory, we deduce that
$b(h)=\Ex^\psi[H_{max}(\ct) > h]$, where $b(h)$ solves \reff{bh}. An easy
extension of \cite{Duquesne2005a} from real trees to w-trees gives the following
result. 

\begin{theorem}[\cite{Duquesne2005a}] \label{Defla} 
Let $\psi$ be a (sub)critical branching mechanism satisfying Assumptions 1 and
2. There exists a $\ct$-measure valued process $(\ell^a, a\geq 0)$
càdlàg for the weak topology on finite measure on $\ct$ such that
$\N^\psi\text{-a.e.}$:
\begin{equation} 
\label{eq:int-la}
\mathbf{m}^{\ct}(dx) = \int_0^\infty \ell^a(dx) da,
\end{equation} 
$\ell^0=0$, $\inf\{a > 0 ; \ell^a = 0\}=\sup\{a \geq 0 ; \ell^a\neq
0\}=H_{\text{max}}(\ct)$ and for every fixed $a\ge 0$,
$\N^\psi\text{-a.e.}$: 
 \begin{itemize}
 \item $\ell^a$ is supported on
$\ct(a)$,
 \item We have for every bounded
continuous function $\phi$ on $\ct$:
\begin{align} 
\langle\ell^a,\phi \rangle 
& = \lim_{\epsilon \downarrow 0}
\frac{1}{b(\epsilon)} \int \phi(x) \ind_{\{h(\ct')\ge
\epsilon\}} \cn_a^{\ct}(dx, d\ct') \\
\label{LaMes}
 & = \lim_{\epsilon \downarrow 0} \frac{1}{b(\epsilon)} \int \phi(x)
\ind_{\{h(\ct')\ge \epsilon\}}
\cn_{a-\epsilon}^{\ct}(dx, d\ct'),\ \text{if}\ a>0.
\end{align}
 \end{itemize}
 Furthermore, we have the branching property: for every $a>0$, the conditional
distribution of the point measure $\cn_a^{\ct}(dx,d\ct')$ under
$\Ex^\psi[d\ct|H_{\text{max}}(\ct)>a]$, given $\pi_a(\ct)$, is that of a Poisson
point measure on $\ct(a)\times \EsAr$ with intensity
$\ell^a(dx)\Ex^\psi[d\ct']$.
\end{theorem}

The measure $\ell^a$ will be called the \emph{local time} measure of $\ct$ at
level $a$. In the case of L\'evy trees, it can also be defined as the image of
the measure $d_sL_s^a(H)$ by the canonical projection $p^H$ (see
\cite{Duquesne2002}), so the above statement is in fact the translation of the
excursion theory of the height process in terms of real trees. This definition
shows that the local time is a function of the tree $\ct$ and does not depend on
the choice of the coding height function. It should be noted that Equation
(\ref{LaMes}) implies that $\ell^a$ is measurable with respect to the
$\sigma$-algebra generated by $\pi_a(\ct)$. 

The next theorem, also from \cite{Duquesne2005a}, relates the discontinuities of
the process $(\ell^a,a\ge 0)$ to the infinite nodes in the tree. Recall
$\mathrm{Br}_\infty(\ct)$ denotes the set of infinite nodes in the Lévy tree
$\ct$.

\begin{theorem} [\cite{Duquesne2005a}] 
Let $\psi$ be a (sub)critical branching mechanism satisfying Assumptions 1 and
2. The set $\{ d(\emptyset,x),\ x\in \mathrm{Br}_\infty(\ct) \}$ coincides
$\Ex^\psi$-a.e. with the set of discontinuity times of the mapping $a\mapsto
\ell^a$. Moreover, $\Ex^\psi$-a.e., for every such discontinuity time $b$, there
is a unique $x_b\in \mathrm{Br}_\infty(\ct)\cap\ct(b)$, and
\[ \ell^b = \ell^{b-} + \Delta_b \delta_{x_b}, \]
where $\Delta_b>0$ is called \emph{mass} of the node $x_b$ and can be obtained
by the approximation 
\begin{equation} \Delta_b = \lim_{\epsilon \rightarrow 0}
\frac{1}{b(\epsilon)}
n(x_b,\epsilon), \label{DefMas} \end{equation}
where $n(x_b,\epsilon)=\int \ind_{\{x=x_b\}}(x)\ind_{\{H_{max}(\ct') >
\epsilon\}}(\ct') \cn_b^\ct(dx,d\ct')$ is the number of sub-trees originating
from $x_b$ with height larger than $\epsilon$. 
\end{theorem}

\subsection{Decomposition of the Lévy tree}

We will frequently use the following notation for the following measure on $\T$:
\begin{equation} 
\label{def:NN0}
 \bN^{\psi}[\ct\in \bullet] = 2\beta \Ex^{\psi}[\ct\in \bullet] +
\int_{(0,+\infty)} r\Pi(dr)\; \Pr_r^{\psi}[\ct
\in \bullet]. 
\end{equation}
where $\psi$ is given by \reff{eq:def-psi}. 
\par The decomposition of a (sub)critical Lévy tree $\ct$ according to a spine
$\llbracket \emptyset, x \rrbracket$, where $x\in \ct$ is a leaf picked at
random at level $a>0$, that is according to the local time $\ell^a(dx)$, is
given in Theorem 4.5 in \cite{Duquesne2005a}. Then by integrating with respect
to $a$, we get the decomposition of $\ct$ according to a spine $\llbracket
\emptyset, x \rrbracket$, where $x\in \ct$ is a leaf picked at random on $\ct$,
that is according to the mass measure $\bm^\ct$. Therefore, we will state this
decomposition without proof.
\par Let $x\in \ct$ and $\{x_i, i\in I_x\}$ the set $\mathrm{Br}(\ct)\cap
\llbracket \emptyset, x \rrbracket$ of branching point on the spine $\llbracket
\emptyset, x \rrbracket$. For $i\in I_x$, we set:
\[
\ct^i=\ct \setminus \left(\ct^{(x,x_i)} \cup
 \ct^{(\emptyset,x_i)}\}\right), 
\]
where $\ct^{(y,x_i)}$ is the connected component of $\ct \setminus\{x_i\}$
containing $y$. We let $x_i$ be the root of $\ct^i$. The metric and measure on
$\ct^i$ are respectively the restriction of $d^\ct$ to $\ct^i$ and the
restriction of $\mathbf{m}^\ct$ to $\ct^i\backslash\{x_i\}$. By construction, if
$x$ is a leaf, we have:
\[
\ct= \llbracket \emptyset, x \rrbracket \circledast_{i\in I_x}(\ct^i,x_i) ,
\]
where $ \llbracket \emptyset, x \rrbracket $ is a w-tree with root $\emptyset$,
metric and mass measure the restrictions of $d^\ct$ and $\mathbf{m}^\ct$ to $
\llbracket \emptyset, x \rrbracket $.

We consider the point measure on $[0, H_x]\times \T$ defined by:
\[
\cm_x=\sum_{i\in i_x} \delta_{(H_{x_i},\ct^i)}.
\]
\begin{theo}[\cite{Duquesne2005a}] 
 \label{theo:T/feuille}
Let $\psi$ be a (sub)critical branching mechanism satisfying Assumptions 1 and
2. We have for
any non-negative measurable function $F$ defined on $[0,+\infty )\times \T$:
\[
\N^\psi\left[\int \bm^\ct(dx) F( H_x , \cm_x)
\right]
=\int_0^{\infty }da \; \expp{-\psi'(0) a }\; \E\left[F\left( a,
 \sum_{i\in I} \ind_{\{z_i\leq a\}} \delta_{(z_i, \bar 
 \ct^i)}\right) \right],
\]
where under $\E$, $\sum_{i\in I} \delta_{(z_i, \bar \ct^i)}(dz, dT)$ is a
Poisson point measure on $[0,+\infty )\times \T$ with intensity $ dz \;
\bN^{\psi}[dT]$.
\end{theo}

\subsection{CSBP process in the Lévy trees}

Lévy trees give a genealogical structure for CSBPs, which is precised in the
next Theorem. We consider the process $\cz=(\cz_a, a\geq 0)$ defined by:
\[
\cz_a=\langle \ell^a,1\rangle.
\]
If needed we will write $\cz_a(\ct)$ to emphasize that $\cz_a $ corresponds to
the tree $\ct$. 

\begin{theorem}[CSBP in Lévy trees, \cite{Duquesne2002} and
 \cite{Duquesne2005a}] 
 \label{CSBPLevy} 
 Let $\psi$ be a (sub)critical branching mechanism satisfying Assumptions 1 and
2, and let $x>0$. The process $\cz$ under $\Pr_x^\psi$ is distributed as the
CSBP $Z$ under $\mathbf{P}^\psi_x$.
 \end{theorem}

\begin{rem}
This theorem can be stated in terms of the height process without Assumption 2.
\end{rem}

\subsection{Super-critical Lévy trees}
\label{sec:Extension}
Let us now briefly recall the construction from \cite{Abraham2009b} for
super-critical Lévy trees using a Girsanov transformation similar to the one
used for CSBPs, see Theorem \ref{GirCSBP}. 

Let $\psi$ be a super-critical branching mechanism satisfying Assumptions 1 and
2. Recall $\theta^*$ is the unique positive root of $\psi'$ and that the
branching
mechanism $\psi_\theta$ is sub-critical if $\theta>\theta^*$, critical if
$\theta=\theta^*$ and super-critical otherwise. We consider the filtration
$\ch=(\ch_a,\ a\ge 0)$, where $\ch_a$ is the $\sigma$-field generated by the
random variable $\pi_a(\ct)$ and the $\Pr_x^{\psi_{\theta^*}}$-negligible sets.
For $\theta \ge \theta^*$, we define the process
$M^{\psi,\theta}=(M_a^{\psi,\theta}, a\geq 0)$ with:
\[ 
M_a^{\psi,\theta} = \exp \Big( \theta x -\theta\cz_a - \psi(\theta)
\int_0^a \cz_s ds \Big)
\]
By absolute continuity of the measures $\Pr_x^{\psi_\theta}$ (resp.
$\Ex^{\psi_\theta}$) with respect to $\Pr_x^{\psi_{\theta^*}}$ (resp.
$\Ex^{\psi_{\theta^*}}$), all the processes $M^{\psi_{\theta},-\theta}$ for
$\theta>\theta^*$ are $\ch$-adapted. Moreover, all these processes are
$\ch$-martingales (see \cite{Abraham2009b} for the proof). Theorem \ref{Defla}
shows that $M^{\psi_{\theta^*},-\theta^*}$ is $\ch$-adapted. Let us now define
the $\psi$-Lévy tree, cut at level $a$ by the following Girsanov transformation.

\begin{definition}
\label{def:Girsanov}
Let $\psi$ be a super-critical branching mechanism satisfying Assumptions 1 and
2. Let $\theta \ge \theta^*$. For $a\ge 0$, we define the distribution
$\Pr_x^{\psi,a}$ (resp. $\Ex^{\psi,a}$) by: if $F$ is a non-negative, measurable
functional defined on $\EsAr$,
 \begin{align}
 \label{eq:GPx}
 \mathbb{E}_x^{\psi,a}[F(\ct)] 
&= \mathbb{E}_x^{\psi_{\theta}} \Big[
M_a^{\psi_\theta,-\theta} F(\pi_a(\ct)) \Big],\\
\label{GirTronc}
 \Ex^{\psi,a}[F(\ct)] 
&= \Ex^{\psi_{\theta}} \Big[
\exp \left( \theta\cz_a +\psi(\theta) \int_0^a \cz_s (ds)
F(\pi_a(\ct) \right)
\Big] .
 \end{align}
\end{definition}
It can be checked that the definition of $\Pr_x^{\psi,a}$ (and of
$\Ex^{\psi,a}$) does not depend on $\theta\ge \theta^*$. 
\par The probability measures $\Pr_x^{\psi,a}$ satisfy a consistence property,
allowing us to define the super-critical Lévy tree in the following way.

\begin{theorem} 
\label{thm:extension}
Let $\psi$ be a super-critical branching mechanism satisfying
assumptions 1 and 2. There exists a probability measure $\Pr_x^{\psi}$ (resp. a
$\sigma$-finite
measure $\Ex^\psi$) on $\EsAr$ such that for $a>0$, we have, if $F$ is a
measurable non-negative functional on $\EsAr$,
 \begin{equation*} \mathbb{E}_x^\psi [F(\pi_a(\ct))] = \mathbb{E}_x^{\psi,a}
[F(\ct)],
\end{equation*}
the same being true under $\Ex^\psi$. 
\end{theorem}
The w-tree $\ct$ under $\P_x^\psi$ or $\N^\psi$ is called a $\psi$-Lévy w-tree
or simply a Lévy tree. 
\begin{proof} For $n\ge 1,\ 0<a_1<...<a_n$, we define a probability measure on
$\EsAr^n$ by:
 \begin{equation*} \Pr_x^{\psi,a_1,...,a_n} (\ct_1\in A_1,...,\ct_n \in A_n) \\
= \Pr_x^{\psi,a_n} ( \ct\in A_n, \pi_{a_{n-1}}(\ct) \in A_{n-1},...,
\pi_{a_1}(\ct) \in A_1 ) \end{equation*}
if $A_1,...,A_n$ are Borel subsets of $\EsAr$. The probability measures
$\Pr_x^{\psi,a_1,...,a_n}$ for $n\ge 1,\ 0<a_1<...<a_n$ then form a projective
family. This is a consequence of the martingale property of
$M^{\psi_\theta,-\theta}$ and the fact that the projectors $\pi_a$ satisfy the
obvious compatibility relation $\pi_b \circ \pi_a = \pi_b$ if $0<b<a$. 
\par By the Daniell-Kolmogorov theorem, there exists a probability measure
$\tilde{\Pr}_x^\psi$ on the product space $\EsAr^{\R_+}$ such that the
finite-dimensional distributions of a $\tilde{\Pr}_x^\psi$-distributed family
are described by the measures defined above. It is easy to construct a version
of a $\tilde{\Pr}_x^\psi$-distributed process that is a.s. increasing. Indeed,
almost all sample paths of a $\tilde{\Pr}_x^\psi$-distributed process are
increasing when restricted to rational numbers. We can then define a w-tree
$\ct^a$ for any $a>0$ by considering a decreasing sequence of rational numbers
$a_n \downarrow a$ and defining $\ct^a = \cap_{n\ge 1} \ct^{a_n}$. Notice that
$\ct^a$ is closed for all $a\in \R_+$. It is easy to check that the
finite-dimensional distributions of this new process are unchanged by this
procedure. Let us then consider $\ct = \cup_{a>0} \ct^a$, endowed with the
obvious metric $d^\ct$ and mass measure $\mathbf{m}$. It is clear that $\ct$ is
a real tree, rooted at the common root of the $\ct^a$. All the $\ct^a$ are
compact, so that $\ct$ is locally compact and complete. The measure $\mathbf{m}$
is locally finite
since all the $\mathbf{m}^{\ct^a}$ are finite measures. Therefore, $\ct$ is a.s.
a w-tree. Then, if we define $\Pr_x^\psi$ to be the distribution of $\ct$, the
conclusion follows. Similar arguments hold under $\N^\psi$.
\end{proof}

\begin{remark} Another definition of super-critical Lévy trees was given by
Duquesne and Winkel \cite{Duquesne2007},\cite{Duquesne2010a}: they consider
increasing families of Galton-Watson trees with exponential edge lengths which
satisfy a certain hereditary property (such as uniform Bernoulli coloring of the
leaves). Lévy trees are then defined to be the Gromov-Hausdorff limits of these
processes. Another approach via backbone decompositions is given in
\cite{Berestycki2009a}. \end{remark}

All the definitions we made for sub-critical Lévy trees then carry over to the
super-critical case. In particular, the level set measure $\ell^a$, which is
$\pi_a(\ct)$-measurable, can be defined using the Girsanov formula. Thanks to
Theorem \ref{GirCSBP}, it is easy to show that the mass process $(\cz_a=\langle
\ell^a ,1 \rangle,\ a\ge 0)$ is under $\Pr_x^\psi$ a CSBP with branching
mechanism $\psi$. In particular, with $u$ defined in \reff{eq:def-u} and $b$ by
\reff{bh}, we have: 
\begin{equation}
 \label{eq:NZ}
\N^\psi\left[1-\expp{-\lambda \cz_a }\right]=u(a,\lambda)
\quad\text{and}\quad 
\N^\psi\left[H_{max}(\ct)>a \right]=\N^\psi
\left[\cz_a >0 \right]= b(a).
\end{equation}
Notice that $b$ is finite only under Assumption 2. 

We set:
\begin{equation}
 \label{eq:s=int-za}
\sigma=\int_0^{+\infty } \cz_a\; da =\mathbf{m}^\ct(\ct)
\end{equation}
for the total mass of the Lévy tree $\ct$. Notice this is consistent with
\reff{eq:int-la} and \reff{eq:s=la} which are defined for (sub)critical Lévy
trees. Thanks to \reff{eq:s=int-za}, notice that $ \sigma$ is distributed as the
total population size of a CSBP with branching mechanism $\psi$. In particular,
its Laplace transform is given for $\lambda> 0$ by:
\begin{equation} 
\label{eq:Ns}
\Ex^\psi [1-\expp{-\lambda \sigma}] =
\psi^{-1}(\lambda).
\end{equation}
Notice that
$\N^\psi[\sigma=+\infty ]=\psi^{-1}(0)> 0$. 

We recall the following
Theorem, from \cite{Abraham2009b}, which sums up the situation for any
branching mechanisme $\psi$. 
\begin{theorem}[\cite{Abraham2009b}] Let $\psi$ be any branching
 mechanism satisfying Assumptions 1 and 2, and let $q>0$ such that
 $\psi(q)\ge 0$. Then, the probability measure $\Pr_x^{\psi_q}$ on
 $\T$ is absolutely continuous w.r.t. $\Pr_x^\psi$, with
\begin{equation} 
\frac{d\Pr_x^{\psi_q}}{d\Pr_x^\psi} =M_\infty^{\psi,q} =
\expp{qx-\psi(q)\sigma} \ind_{\{\sigma<+\infty \}}. 
\end{equation}
Similarly, the excursion measure $\Ex^{\psi_q}$ on $\T$ is absolutely
continuous w.r.t. $\Ex^\psi$ and we have
\begin{equation}
 \label{Girsanov} \frac{d\Ex^{\psi_q}}{d\Ex^\psi} =
 \expp{-\psi(q)\sigma} \ind_{\{\sigma<+\infty \}}.
\end{equation}
\end{theorem}

When applying Girsanov formula (\ref{Girsanov}) to $q=\bar\theta$
defined by \reff{eq:def-barq}, we get the following remarkable
Corollary, due to the fact that $\psi_{\theta}(\bar\theta-\theta) = 0$.

\begin{cor}
 Let $\psi$ be a critical branching mechanism satisfying Assumptions 1 and 2,
and $\theta\in \Theta^\psi$ with $\theta<0$. Let $F$ be a non-negative
measurable functional defined on $\EsAr$. We have:
\begin{align}
\nonumber
\expp{(\bar{\theta}-\theta)x} \; \Er_x^{\psi_\theta}[F(\ct)\ind_{\{\sigma
 <+\infty\}} 
]& = \Er_x^{\psi_{\bar\theta}}[F(\ct)], \\
\label{GirThetaBar} 
 \Ex^{\psi_\theta}[F(\ct)\ind_{\{\sigma <+\infty\}}] & = 
\Ex^{\psi_{\bar\theta}}[F(\ct)].
\end{align}
\end{cor}

We deduce from Proposition \ref{Buisson} and Theorem \ref{thm:extension} that
$\cn_0^{\ct}(dx,d\ct') $ defined by \reff{eq:branchement} with $a=0$ is under
$\P^\psi_x(d\ct)$ a Poisson point measure on $\{\emptyset\}\times \T$ with
intensity $\sigma \delta_\emptyset(dx) \N^\psi[d\ct']$. Then we deduce from
\reff{eq:GPx}, with $F=1$, that for $\theta\geq \theta^*$:
\begin{equation}
 \label{eq:N-q}
\Ex^{\psi_{\theta}} \left[1-\exp\left(\theta\cz_a +\psi(\theta) \int_0^a \cz_s
 ds\right)\right] = -\theta. 
\end{equation}

\subsection{Pruning Lévy trees}
\label{Pruning}

We recall the construction from \cite{Abraham2010a} on the pruning of Lévy
trees. Let $\ct$ be a random Lévy w-tree under $\Pr_x^\psi$ (or under
$\Ex^\psi$), with $\psi$ conservative. Let
\[
m^{(\text{ske})}(dx,d\theta) = \sum_{i\in I^{\text{ske}}}
\delta_{(x_i,\theta_i)}(dx,d\theta) 
\]
 be, conditionally on $\ct$, a Poisson point measure on $\ct\times\R_+$ with
intensity $2\beta l^{\ct}(dx) d\theta$. Since there is a.s. a countable number
of branching points (which have $l^{\ct}$-measure 0), the atoms of this measure
are distributed on $\ct \setminus(\mathrm{Br}(\ct) \cup \mathrm{Lf}(\ct))$. \par
If $\Pi=0$, we have $\mathrm{Br}_\infty(\ct)=\emptyset$ a.s. whereas if
$\Pi(\R_+)=\infty$, $\mathrm{Br}_\infty(\ct)$ is a.s. a countable dense subset
of $\ct$. If the latter condition holds, we consider, conditionally on $\ct$, a
Poisson point measure 
\[
m^{(\text{nod})}(dx,d\theta) = \sum_{i\in I^{\text{nod}}}
\delta_{(x_i,\theta_i)}(dx,d\theta) 
\]
 on $\ct\times\R_+$ with intensity $$\sum_{y\in
\mathrm{Br}_\infty(\ct)}\Delta_y\delta_y(dx)\,d\theta$$ where $\Delta_x$ is the
mass of the node $x$, defined by \reff{DefMas}. Hence, if $\theta>0$, a node
$x\in \mathrm{Br}_\infty(\ct)$ is an atom of $m^{(\text{nod})}(dx,[0,\theta])$
with probability $1-\exp(-\theta\Delta_x)$. The set
\[
\{x_i,\ i\in
I^{\text{nod}}\}=\left\{x\in\ct,\ m^{(\text{nod})}
 \bigl(\{x\}\times\R_+\bigr)>0\right\}
\]
of marked branching points corresponds $\P_x^\psi$-a.s or $\N^\psi$-a.e. to
$\mathrm{Br}_\infty(\ct)$. For $i\in I^{\text{nod}}$, we set
\[
\theta_i=\inf\left\{\theta>0,\ m^{(\text{nod})}\bigl(\{x_i\}\times
[0,\theta]\bigr)>0\right\}
\]
the first mark on $x_i$ (which is conditionally on $\ct$ exponentially
distributed with parameter $\theta_{x_i}$), and we set
\[
\{\theta_j,\ j\in J_i^{\text{nod}}\}=\left\{\theta>\theta_i, \
 m^{(\text{nod})}\bigl(\{x_i\}\times\{\theta\}\bigr)>0\right\}
\]
so that we can write
\[
 m^{(\text{nod})}(dx,d\theta) = \sum_{i\in I^{\text{nod}}}
 \delta_{x_i}(dx) \left( 
\delta_{\theta_i}(d\theta) + \sum_{j\in J^{\text{nod}}_i}
\delta_{\theta_j}(d\theta)\right) .
\]

We set the measure of marks:
\begin{equation}
 \label{eq:def-M}
\cm(dx,d\theta)=m^{(\text{ske})}(dx,d\theta) + m^{(\text{nod})}(dx,d\theta), 
\end{equation}
and consider the family of w-trees $\Lambda(\ct,\cm)=(\Lambda_\theta(\ct,\cm)
,\theta \ge 0)$, where the $\theta$-pruned w-tree $\Lambda_\theta$ is defined
by:
\begin{equation*} 
 \Lambda_\theta(\ct,\cm) = \{ x\in \ct,\ \cm (\llbracket
\emptyset , x \llbracket \times [0,\theta]) = 0 \},
\end{equation*}
rooted at $\emptyset^{\Lambda_\theta(\ct,\cm) } = \emptyset^\ct$, and the metric
$d^{\Lambda_\theta(\ct,\cm)} $ and the mass measure
$\bm^{\Lambda_\theta(\ct,\cm)} $ are the restrictions of $d^\ct$ and $\bm^\ct$
to $\Lambda_\theta(\ct,\cm) $. In particular, we have $\Lambda_0( \ct,\cm)=\ct$.
The family of w-trees $\Lambda(\ct,\cm)$ is a non-increasing family of real
trees, in a sense that $ \Lambda_{\theta'}(\ct,\cm)\supset
\Lambda_\theta(\ct,\cm) $ for $0\leq \theta'\leq \theta$, see Figure
\ref{fig:arbres}. In particular, we have that the pruning operators satisfy a
cocycle property, for $\theta_1\geq 0$ and $\theta_2\geq 0$:
\[
\Lambda_{\theta_2}\bigl(\Lambda_{\theta_1}(\ct,\cm),\cm_{\theta_1}\bigr)
= \Lambda_{\theta_2+\theta_1}(\ct,\cm),
\]
where $\cm_{\theta}(A\times [0,q])=\cm(A\times [\theta, \theta+q])$. Abusing
notation, we write $\N^\psi(d\ct,d\cm)$ for the distribution of the pair
$(\ct,\cm)$ when $\ct$ is distributed according to $\N^\psi(d\ct)$ and
conditionally on $\ct$, $\cm$ is distributed as described above.
\par The following result can be deduced from \cite{Abraham2009b}.
\begin{theorem}
\label{ProcArbre*} 
Let $\psi$ be a branching mechanism satisfying Assumptions 1 and 2. There exists
a non-increasing $\T$-valued Markov process $(\ct_\theta, \theta\in
\Theta^\psi)$ such that for all $q\in \Theta^\psi$, the process
$(\ct_{\theta+q}, \theta\geq 0)$ is distributed as $\Lambda(\ct,\cm)$ under
$\Ex^{\psi_q}[d\ct,d\cm]$.
\end{theorem}
In particular, this Theorem implies that $\ct_\theta$ is distributed as
$\Ex^{\psi_\theta}$ for $\theta\in \Theta^\psi$ and that for $\theta_0 \ge 0$,
under $\Ex^\psi$, the process of pruned trees $(\Lambda_{\theta_0+\theta}(\ct),
\theta \ge 0)$ has the same distribution as $(\Lambda_\theta(\ct),\theta\ge 0)$
under $\Ex^{\psi_{\theta_0}}[d\ct]$. 
\par We want to study the time-reversed process $(\ct_{-\theta}, \theta\in
-\Theta^\psi)$, which can be seen as a growth process. This process grows by
attaching sub-trees at a random point, rather than slowly growing uniformly
along the branches. We recall some results from \cite{Abraham2009b} on the
growth process. From now on, we will assume in this Section that the
\textbf{branching mechanism $\psi$ is critical}, so that $\psi_\theta$ is
sub-critical iff $\theta>0$ and super-critical iff $\theta<0$.

\begin{figure}[htbp]
 \begin{center}
 \includegraphics{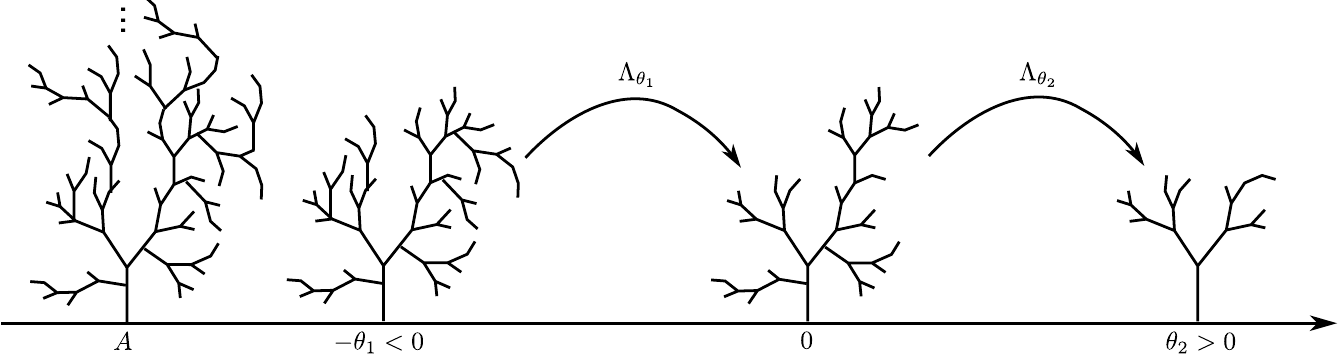}
\caption{The pruning process, starting from explosion time $A$
 defined in \reff{eq:def-A}.
}
\label{fig:arbres}
 \end{center}
\end{figure}

We will use the following notation for the total mass of the tree $\ct_\theta$
at time $\theta\in \Theta^\psi$:
\begin{equation}
 \label{eq:def-sq}
\sigma_\theta=\bm^{\ct_\theta} (\ct_\theta). 
\end{equation}
The total mass process $(\sigma_\theta, \theta\in \Theta^\psi)$ is a pure-jump
process taking values in $(0,+\infty ]$. 

\begin{lemma}[\cite{Abraham2009b}] \label{lem:CroiSig} Let $\psi$ be a
 critical branching mechanism satisfying Assumptions 1 and 2. If $0 \le
 \theta_2 < \theta_1$, then we have:
\[
 \Ex^\psi [\sigma_{\theta_2} | \ct_{\theta_1} ]
= \sigma_{\theta_1} \frac{\psi'(\theta_1)}{\psi'(\theta_2)}\cdot
\]
\end{lemma}

Consider the ascension time (or explosion time):
\begin{equation}
 \label{eq:def-A}
A = \inf \{ \theta \in \Theta^\psi,\ \sigma_\theta < \infty \},
\end{equation}
where we use the convention $\inf \emptyset = \theta_\infty$. The following
Theorem gives the distribution of the ascension time $A$ and the distribution of
the tree at this random time. Recall that $\bar \theta=\psi^{-1}(\psi(\theta))$
is defined in \reff{eq:def-barq}. 

\begin{theorem}[\cite{Abraham2009b}] 
\label{ADAvExplo} 
Let $\psi$ be a critical branching mechanism satisfying Assumptions 1 and
2.
\begin{enumerate} 
\item For all $\theta\in \Theta^\psi$, we have $ \Ex^{\psi} [ A > \theta
 ] = \bar \theta - \theta$.
\item If $\theta_\infty < \theta<0$, under $\Ex^\psi$, we have, for any
 non-negative measurable functional $F$,
\begin{equation*}
 \Ex^\psi[ F(\ct_{A+\theta'},\theta'\ge 0) | A=\theta ] =
\psi'(\bar\theta) \Ex^\psi [F(\ct_{\theta'},\theta'\ge 0) \sigma_0
\expp{-\psi(\theta)\sigma_0} ].
\end{equation*} 
\item For all $\theta\in \Theta^\psi$, we have $
 \N^\psi[\sigma_A<+\infty |A=\theta]=1$. 
\end{enumerate} 
\end{theorem}

In other words, at the ascension time, the tree can be seen as a size-biased
critical Lévy tree. A precise description of $\ct_A$ is given in
\cite{Abraham2009b}. Notice that in the setting of \cite{Abraham2009b}, there is
no need of Assumption 2. 

\section{The growing tree-valued process}
\label{sec:growing}
\subsection{Special Markov Property of pruning}

In \cite{Abraham2010a}, the authors prove a formula describing the structure of
a Lévy tree, conditionally on the $\theta$-pruned tree obtained from it in the
(sub)critical case. We will give a general version of this result. From the
measure of marks, $\cm$ in \reff{eq:def-M}, we define a measure of increasing
marks by: 
\begin{equation}
\label{mfleche}
 \cm^\uparrow (dx,d\theta') = \sum_{i\in I^\uparrow}
\delta_{(x_i,\theta_i)}(dx,d\theta'),
\end{equation}
with
\[
I^\uparrow=\left\{i\in I^\text{ske}\cup I^\text{nod}; \cm(\llbracket
 \emptyset,x_i \rrbracket \times [0,\theta_i]) = 1\right\}. 
\]
The atoms $(x_i,\theta_i)$ for $ i\in I^\uparrow$ correspond to marks such that
there are no marks of $\cm$ on $\llbracket \emptyset,x_i \rrbracket$ with a
$\theta$-component smaller than $\theta_i$. In the case of multiple $\theta_j$
for a given node $x_i\in \mathrm{Br}_\infty(\ct)$, we only keep the smallest
one. In the case $\Pi=0$, the measure $\cm^\uparrow$ describes the jumps of a
record process on the tree, see \cite{Abraham2011} for further work in this
direction. The $\theta$-pruned tree can alternatively be defined using
$\cm^\uparrow $ instead of $\cm$ as for $\theta\geq 0$:
\[
\Lambda_\theta(\ct,\cm) = \{ x\in \ct,\ \cm^\uparrow (\llbracket
\emptyset , x \llbracket \times [0,\theta]) = 0 \}. 
\]
We set:
\[
I_\theta^{\uparrow}
=\left\{i\in I^\uparrow, x_i\in
 \mathrm{Lf}(\Lambda_\theta(\ct, \cm))\right\}
=\left\{i\in I^\uparrow, \theta_i<\theta\quad\text{and}\quad
\cm^\uparrow(\llbracket \emptyset,x_i \llbracket \times
[0,\theta]) = 0\right\}
\]
and for $i\in I_\theta^{\uparrow}$:
\[
 \ct^i = \ct\backslash \ct^{\emptyset, x_i}=\{ x\in \ct,\ x_i \in
 \llbracket \emptyset,x\rrbracket \},
\]
where $\ct^{y,x}$ is the connected component of $\ct\backslash\{x\}$ containing
$y$. For $i\in I_\theta^{\uparrow}$, $\ct^i$ is a real tree, and we will
consider $x_i$ as its root. The metric and mass measure on $\ct^i$ are the
restriction of the metric and mass measure of $\ct$ on $\ct^i$. By construction,
we have:
\begin{equation}
 \label{eq:T-from-LT}
\ct=\Lambda_\theta(\ct,\cm) \circledast_{i\in I_\theta^{\uparrow}}
(\ct^i,x_i). 
\end{equation}

Now we can state the general special Markov property.

\begin{theorem}[Special Markov Property]
\label{SMP} 
Let $\psi$ be a branching mechanism satisfying Assumptions 1 and 2. Let $\theta
>0$. Conditionally on $\Lambda_\theta(\ct,\cm)$, the point measure: 
\[
 \cm_\theta^{\uparrow}(dx,d\ct',d\theta') = \sum_{i\in I_\theta^{\uparrow}}
\delta_{(x_i,\ct^i,\theta_i)}(dx,d\ct',d\theta') 
\]
under $\P^\psi_{r_0}$ (or under $\Ex^\psi$) is a Poisson point measure on
$\Lambda_\theta(\ct,\cm)\times \EsAr \times (0,\theta]$ with intensity: 
\begin{equation}
\label{Intensite} 
 \mathbf{m}^{\Lambda_\theta(\ct,\cm)}(dx) \Big( 2\beta
 \Ex^\psi[d\ct'] + \int_{(0,+\infty)} \Pi(dr) \; r \expp{-\theta' r}
 \Pr^\psi_r(d\ct') \Big)\; \ind_{(0,\theta]}(\theta')\; d\theta'.
\end{equation}
\end{theorem}

\begin{proof} 
 It is not difficult to adapt the proof of the special Markov property in
\cite{Abraham2010a} to get Theorem \ref{SMP} in the (sub)critical case by taking
into account the pruning times $\theta_i$ and the w-tree setting; and we omit
this proof which can be found in \cite{Abraham2012a}. We prove how to extend the
result to the super-critical Lévy trees
using the Girsanov transform of Definition \ref{def:Girsanov}.
\par Assume that $\psi$ is super-critical. For $a>0$, we shall write
$\Lambda_{\theta,a} (\ct,\cm)=\pi_a(\Lambda_\theta(\ct,\cm))$ for short.
According to \reff{eq:T-from-LT} and the definition of super-critical Lévy
trees, we have that for any $a>0$, the truncated tree $\pi_a(\ct)$ can be
written as:
\[ 
\pi_a(\ct) = \Lambda_{\theta,a}(\ct,\cm) \circledast_{\begin{subarray}
 \text{ } i\in
I_\theta^{\uparrow}, \\ 
H_{x_i}\leq a
\end{subarray}}
(\pi_{a-H_{x_i}}(\ct^i),x_i) 
\]
and we have to prove that $\sum_{i\in I_\theta^{\uparrow}}
\delta_{(x_i,\ct^i,\theta_i)}(dx,d\ct',d\theta')$ is conditionally on
$\Lambda_{\theta}(\ct,\cm) $ a Poisson point measure with intensity
\reff{Intensite}. Since $a$ is arbitrary, it is enough to prove that the point
measure $\cm_a$, defined by
\[
 \cm_a(dx,d\ct',d\theta') = \sum_{i\in
 I_\theta^{\uparrow}} 
\ind_{\{ H_{x_i} \le a \}} \;
\delta_{(x_i,\pi_{a-H_{x_i}}(\ct^i),\theta_i)} (dx, d\ct',d\theta'),
\]
is conditionally on $\Lambda_{\theta,a}(\ct,\cm) $ a Poisson point measure with
intensity :
\begin{multline}
 \label{eq:intensite-ma}
\ind_{[0,a]}(H_x)\; \bm ^{\Lambda_{\theta}(\ct,\cm)}(dx)\; 
\ind_{(0,\theta]}(\theta')\; d\theta' \\
 \Big( 2\beta
(\pi_{a-H_x})_* \Ex^\psi(d\ct')+ \int_{(0,+\infty)} \Pi(dr) \; r
 \expp{-\theta' r} 
 (\pi_{a-H_x})_*\Pr^\psi_r(d\ct') \Big) .
\end{multline}

Recall $\theta^*$ is the unique real number such that
$\psi_{\theta^*}'(0)=0$, that is, such that $\psi_{\theta^*}$ is
critical. Let $\Phi$ be a non-negative, measurable functional on
$\Lambda_{\theta,a}(\ct,\cm)\times \EsAr \times (0,\theta]$ and let $F$
be a non-negative measurable functional on $\EsAr$. Let
\[
B = \Ex^\psi [ F(\Lambda_{\theta,a}(\ct,\cm)) \exp(-\left\langle \cm_a,\Phi
\right\rangle) ] .
\]
Thanks to Girsanov formula \reff{GirTronc} and the special Markov property for
critical branching mechanisms, we get:
\begin{align*} 
B 
&= \Ex^{\psi_{\theta^*}}
\left[F(\Lambda_{\theta,a}(\ct,\cm)) 
\exp(-\left\langle \cm_a,\Phi \right\rangle)
\exp\left (\theta^* \cz_a(\ct) + \psi(\theta^*) \int_0^a \cz_h(\ct) dh
\right) \right] \\
& = \Ex^{\psi_{\theta^*}}\Big[
F(\Lambda_{\theta,a}(\ct,\cm)) \exp
\Big(\theta^* \cz_a(\Lambda_\theta(\ct,\cm)) + \psi(\theta^*) \int_0^a
\cz_h(\Lambda_\theta(\ct,\cm)) dh \Big) \\
&\hspace{6cm} \exp \Big(-\int \bm^{\Lambda_{\theta,a}(\ct,\cm)}(dx)
G(H_x,x,\theta ) \Big) 
\Big], 
\end{align*}
with $\bm^{\Lambda_{\theta,a}(\ct,\cm)}(dx)= \ind_{[0,a]}(H_x)\; \bm
^{\Lambda_{\theta}(\ct,\cm)}(dx) $ and $G(h,x,\theta)$ equal to:
\begin{multline*} 
\int_0^\theta d\theta' \; \Big \{ 2\beta \Ex^{\psi_{\theta^*}} \left[
1-\exp \left(- \Phi(x,\pi_{a-h}(\ct),\theta') + \theta^* \cz_{a-h} (\ct)
+\psi(\theta^*)\int_0^{a-h} \cz_t(\ct)
dt \right) \right] \\ 
+ \int_{(0,+\infty)}\!\!\!\!\!\!\!\!\!\!\! \Pi_{\theta^*} (dr) r
\expp{-\theta ' 
r}\Er_r^{\psi_{\theta^*}} \!\left[1-\exp (-
\Phi(x,\pi_{a-h}(\ct),\theta') + \theta^* 
\cz_{a-h} (\ct) +\psi(\theta^*)\int_0^{a-h}\!\!\!\!\!\!\!\!\! \cz_t(\ct) dt )
\right]\Big\}. 
\end{multline*}
By using the Poisson decomposition of $\Pr_r^{\psi_{\theta^*}}$ (Proposition
\ref{Buisson}), we see that $G(h,x,\theta)$ can be written as:
\[
G(h,x,\theta ) = \int_0^\theta d\theta' \Big\{2\beta g(h,x,\theta ') +
\int_{(0,\infty)} \Pi_{\theta^*} (dr)\ r \expp{-\theta ' r}\left(1- \exp(-r
g(h,x,\theta '))\right)\Big\},
\]
with 
\[
g(h,x,\theta ') = \Ex^{\psi_{\theta^*}} \left[
1-\exp \left(-
\Phi(x,\pi_{a-h}(\ct),\theta') +\theta^* \cz_{a-h} (\ct)
+\psi(\theta^*)\int_0^{a-h} \cz_t(\ct)
dt\right) \right].
\]
Thanks to the Girsanov formula and \reff{eq:N-q}, we get:
\begin{align*} 
g(h,x,\theta ') 
&= \Ex^{\psi_{\theta^*}} \Big[
 (1-\exp (- \Phi(x,\pi_{a-h}(\ct),\theta'))) \exp\Big( \theta^* \cz_{a-h} (\ct)
 +\psi(\theta^*)\int_0^{a-h}
 \cz_t(\ct)dt \Big) \Big] \\
 & \hspace{2cm} + \Ex^{\psi_{\theta^*}}\Big[1-\exp\Big(\theta^* \cz_{a-h} (\ct)
 +\psi(\theta^*)\int_0^{a-h} \cz_t(\ct)) dt \Big) \Big]\\
&= \Ex^{\psi} \Big[ 1-\exp
(-\Phi(x,\pi_{a-h}(\ct),\theta')) \Big]- \theta^*.
 \end{align*}
With $\tilde g(h,x,\theta ')= \Ex^{\psi} \Big[ 1-\exp
(-\Phi(x,\pi_{a-h}(\ct),\theta')) \Big]$ and thanks to \reff{eq:pi-q}, we get:
\begin{multline*}
 G(h,x,\theta ) 
=\int_0^\theta d\theta' \Big\{2\beta \tilde g(h,x,\theta ') +
\int_{(0,\infty)} \Pi (dr)\ r \expp{-\theta ' r}\left(1- \exp(-r
\tilde g(h,x,\theta '))\right)\Big\}\\
+\psi(\theta^*) - \psi_\theta(\theta^*).
\end{multline*}
Notice that from the definition of $G$ we have $g$ replaced by $\tilde g$,
$\Pi_{\theta^*}$ replaced by $\Pi$ and the additional term $\psi(\theta^*) -
\psi_\theta(\theta^*)$. As $ \int \bm^{\Lambda_{\theta,a}(\ct,\cm)}(dx) =
\int_0^a \cz_h(\Lambda_\theta(\ct))dh $, we get:
\begin{multline}
 \label{eq:ARH}
B = \Ex^{\psi_{\theta^*}}\! \left[
F(\Lambda_{\theta,a}(\ct,\cm))R(\Lambda_{\theta,a}(\ct,\cm)) \right. \\
 \left. \exp
\Big(\theta^* \cz_a(\Lambda_\theta(\ct,\cm)) + \psi_\theta(\theta^*)
\int_0^a \cz_h(\Lambda_\theta(\ct,\cm)) dh \Big)\!\right],
\end{multline}
with 
\begin{multline}
R(\ct)=\exp \Big(-\int \bm^\ct (dx) 
\int_{0}^{\theta} d\theta'\Big[2\beta
\tilde g(H_x,x,\theta ') + \\
\int_{(0,\infty)} \Pi(dr)\ r \expp{-\theta ' r}\left(1- \exp(-r
\tilde g(H_x,x,\theta '))\right)
\Big] \Big).
\end{multline}
Taking $\Phi=0$ (and thus $R=1$) in \reff{eq:ARH} yields:
\begin{multline}
 \label{eq:AFH}
\Ex^\psi [ F(\Lambda_{\theta,a}(\ct,\cm)) ] \\
=\Ex^{\psi_{\theta^*}}\left[
F(\Lambda_{\theta,a}(\ct,\cm)) \exp
\Big(\theta^* \cz_a(\Lambda_\theta(\ct,\cm)) + \psi_\theta(\theta^*) \int_0^a
\cz_h(\Lambda_\theta(\ct,\cm)) dh \Big) \right].
\end{multline}
Using \reff{eq:AFH} with $F$ replaced by $FR$ gives:
\[
\Ex^\psi\Big[ \exp(-\langle \cm_a , \Phi
\rangle) F(\Lambda_{\theta,a}(\ct,\cm)) \Big]= B =
\N^\psi\left[F(\Lambda_{\theta,a}(\ct,\cm))R(\Lambda_{\theta,a}(\ct,\cm))
\right]. 
\]
This implies that $\cm_a$ is, conditionally on $\Lambda_{\theta,a}(\ct,\cm)$, a
Poisson point measure with intensity \reff{eq:intensite-ma}. This ends the
proof. 
\end{proof}

\subsection{An explicit construction of the growing process}
\label{sec:growing-process}

In this section, we will construct the growth process using a family of Poisson
point measures. Let $\psi$ be a branching mechanism satisfying Assumptions 1 and
2. Let $\theta\in \Theta^\psi$. According to \reff{def:NN0} and \reff{eq:pi-q},
we have:
\begin{equation} 
\label{def:NN}
 \bN^{\psi_\theta}[\ct\in \bullet] = 2\beta \Ex^{\psi_\theta}[\ct\in
 \bullet] + 
\int_{(0,+\infty)} \Pi(dr) r \expp{-\theta r} \Pr_r^{\psi_\theta}(\ct
\in \bullet). 
\end{equation}

Let $\ct^{(0)}\in\T$ with root $\emptyset$. For $q\in\Theta^\psi$ and $q\le
\theta$, we set:
\[
\at_q^{(0)}=\ct^{(0)}\quad\mbox{and}\quad \bm_q^{(0)}=\bm^{\ct^{(0)}}.
\]
We define the w-trees grafted on $\ct^{(0)}$ by recursion on their generation.
We suppose that all the random point measures used for the next construction are
defined on $\EsAr$ under a probability measure $Q^{\ct^{(0)}}(d\omega)$.
\par Suppose that we have constructed the family of trees and mass measures
$((\at_q^{(k)},\bm_q^{(n)}), 0\leq k\leq n, q \in \Theta^\psi\cap (-\infty
,\theta))$. We write
$$\at^{(n)}=\bigsqcup_{q\in\Theta^\psi,\ q\le \theta}\at_q^{(n)}.$$
We define the ($n+1$)-th generation as follows. Conditionally on all trees from
generations smaller than $n$, $(\at_q^{(k)},\ 0\leq k\leq n,\ q \in
\Theta^\psi\cap (-\infty ,\theta))$, let 
\[
\cn^{n+1}_{\theta}(dx,d\ct,dq)=\sum_{j\in J^{(n+1)}}
\delta_{(x_j,\ct^j,\theta_j)}(dx,d\ct,dq) 
\] 
be a Poisson point measure on $ \at^{(n)} \times \T\times\Theta^\psi$ with
intensity:
\[
\mu^{n+1}_\theta(dx,d\ct,dq)=
\bm^{(n)}_q(dx) \bN^{\psi_q}[d\ct]\; \ind_{\{ q\leq\theta\}}\; dq.
\]
For $q\in\Theta^\psi$ and $q\le\theta$, we set
$$J_q^{(n+1)}=\{j\in J^{(n+1)},\ q<\theta_j\}$$
and we define the tree $\at_q^{(n+1)}$ and the mass measure $\bm_q^{(n+1)}$ by:
\[
\at_q^{(n+1)}=\at_q^{(n)}\circledast_{j\in
 J_q^{(n+1)}}(\ct^j,x_j)\quad \mbox{and}\quad \bm_q^{(n+1)}=\sum_{j\in
 J_q^{(n+1)}}\bm^{\ct^j}(dx).
\]

Notice that by construction, $(\at_q^{(n)}, n\in \N)$ is a non-decreasing
sequence of trees. We set $\at_q$ the completion of $\cup _{n\in \N}
\at_q^{(n)}$, which is a real tree with root $\emptyset$ and obvious metric
$d^{\at_q}$, and we define a mass measure on $\at_q$ by $\bm^{\at_q}= \sum_{n\in
\N} \bm^{(n)}_q$.
\par For $q\in \Theta^\psi$ and $q<\theta$, we consider $\cf_{q}$ the
$\sigma$-field generated by $\at^{(0)}$ and the sequence of
random point measures $(\ind_{\{q'\in [q,\theta
]\}}\cn_{\theta}^{(n)}(dx,d\ct,dq') , n\in \N)$. We set $\cn_{\theta}=\sum_{n\in
\N} \cn^n_{\theta}$. The backward random point process $q\mapsto \ind_{\{q\leq
q'\}} \cn_\theta(dx, d\ct, dq')$ is by construction adapted to the backward
filtration $(\cf_q, q\in \Theta^\psi\cap (-\infty ,\theta])$.
\par The proof of the following result is postponed to Section
\ref{sec:proof-thm}. 

\begin{theo}
 \label{theo:at=ct}
 Let $\psi$ be a branching mechanism satisfying Assumptions 1 and
2. Under $Q^{\psi_\theta}:=\N^{\psi_\theta}[d\ct^{(0)}]Q^{\ct^{(0)}}(d\omega)$,
the process 
\[ (( \at_q,d^{\at_q}, \emptyset, \bm^{\bar \at_q}), q\in \Theta^\psi\cap
(-\infty ,\theta])\]
 is a $\T$-valued backward Markov process with respect to the backward
filtration $\cf^{\theta}= (\cf_q, q\in \Theta^\psi\cap (-\infty ,\theta])$. It
is distributed as $((\ct_q,\bm^{\ct_q}), q\in \Theta^\psi\cap (-\infty
,\theta])$ under $\N^\psi$.
\end{theo}
Notice the Theorem in particular entails that $( \at_q,d^{\at_q}, \emptyset,
\bm^{\bar \at_q})$ is a w-tree for all $q$. 
\par We shall use the following Lemma. 
\begin{lem}
Let $\psi$ be a branching mechanism satisfying Assumptions 1 and
2. Let $K$ be a measurable non-negative process (as a function of $q$) defined
on
$\R_+\times\T\times\T$ which is predictable with respect to the backward
filtration $\cf^{\theta}$. We have:
\[
 Q^{\psi_\theta}\left[\int \cn_\theta(dx, d\ct,dq)\;
 K(q, \at_{q}, \at_{q-}) \right]
 =Q^{\psi_\theta}\left[\int K\Big(q, \at_q ,\at_q \circledast
 (\ct,x)\Big) \; 
 \mu_\theta(dx,d\ct,dq) \right],
\]
where $\mu_\theta(dx,d\ct,dq)= \sum_{n\in \N^*}
\mu^{n}(dx,dT,dq)=\bm^{\at_q}(dx) \bN^{\psi_q}[d\ct] \; \ind_{\{q\in
\Theta^\psi, q\leq \theta\}}\; dq$. 
\end{lem}

This means that the predictable compensator of $\cn_\theta$ is given by:
\[
\mu_\theta(dx,d\ct,dq)=\bm^{\at_q}(dx)
\bN^{\psi_q}[d\ct]\; \ind_{\{q \in \Theta^\psi,
 q\leq\theta\}}\; dq .
\]
Notice this construction does not fit in the usual framework of random point
measures as the support at time $q$ of the predictable compensator is the
(predictable backward in time) random set $\at_q\times \T\times \Theta^\psi$.

\begin{proof}
Based on the recursive construction, we have:
\begin{multline*}
Q^{\psi_\theta}\left[\int \cn_\theta(dx, d\ct,dq)\;
 K(q, \at_{q}, \at_{q-})
 \right]\\
=\sum_{n=0}^{+\infty}Q^{\psi_\theta}\left[Q^{\psi_\theta}\left[\int
 \cn_\theta^n(dx,d\ct,dq)\; 
 K(q, \at_{q}, \at_{q}\circledast (\ct, x))\Bigm|
 (\at_s^{(k)},\ k\le n,\ s\le \theta)\right]\right].
\end{multline*}
Now, by construction, we have that:
\[
\at_q=\at_q^{(n)}\circledast _{j\in
 J_q^{(n)}}(\tilde\ct_j,x_j)
\]
for $\tilde\ct_j=\at_q\backslash\at_q^{(x_j,\emptyset)}$ which is a
measurable function of $\ind_{\{q'>q\}}\cn_\theta^n(dx, d\ct,dq')$ and of the
point measures $\ind_{\{q'>q\}} \cn_\theta^\ell(dx, d\ct, dq')$ for $\ell\ge
n+1$. Therefore, applying Palm formula with the function
\begin{multline*}
F_n\Big(q,\ct,x,\sum_{j\in
 J^{(n)},q_j>q}\delta_{(x_j,\ct^j,\theta_j)}\Big)
=Q^{\psi_\theta}\left[K\Big(q, \at_q^{(n)}\circledast _{j\in
 J_q^{(n)}}(\tilde\ct_j,x_j),\right.\\
\left.\at_q^{(n)}\circledast _{j\in
 J_q^{(n)}}(\tilde\ct_j,x_j)\circledast(\ct,x)\Big)\Bigm|
 (\at_s^{(k)},\ k\le n,\ s\le \theta),\cn_\theta^n\right],
\end{multline*}
we get:
\begin{align*}
Q^{\psi_\theta} & \left[\int \cn_\theta(dx, d\ct,dq)\;
 K(q, \at_{q}, \at_{q-}) \right]\\
& =\sum_{n=0}^{+\infty}Q^{\psi_\theta}\Big[Q^{\psi_\theta}\Big[\int
 \cn_\theta^n(dx,d\ct,dq)\; \\
&\hspace{4cm} F_n\Big(q,\ct,x,\sum_{j\in
 J^{(n)},q_j>q}\delta_{(x_j,\ct^j,\theta_j)}\Big) \Bigm|
 (\at_s^{(k)},\ k\le n,\ s\le \theta)\Big]\Big]\\
& =\sum_{n=0}^{+\infty}Q^{\psi_\theta}\Big[Q^{\psi_\theta}\Big[\int
 \mu^n_\theta (dx,d\ct,dq)\; \\
&\hspace{4cm} F_n\Big(q,\ct,x,\sum_{j\in
 J^{(n)},q_j>q}\delta_{(x_j,\ct^j,\theta_j)}\Big) \Bigm|
 (\at_s^{(k)},\ k\le n,\ s\le \theta)\Big]\Big]\\
& =\sum_{n=0}^{+\infty}Q^{\psi_\theta}\left[Q^{\psi_\theta}\left[\int
\mu_\theta^n(dx,d\ct,dq)\;
 K\Big(q, \at_q^{(n)}\circledast _{j\in
 J_q^{(n)}}(\tilde\ct_j,x_j),\right.\right.\\
& \qquad\qquad\qquad\qquad \left.\left.\at_q^{(n)}\circledast _{j\in
 J_q^{(n)}}(\tilde\ct_j,x_j)\circledast(\ct,x)\Big)\Bigm|
 (\at_s^{(k)},\ k\le n,\ s\le \theta)\right]\right]\\
& =\sum_{n=0}^{+\infty}Q^{\psi_\theta}\left[\int
\mu_\theta^n(dx,d\ct,dq)\;K\Big(q, \at_{q}, \at_{q}\circledast(\ct,x)\Big)
 \right]\\
& =Q^{\psi_\theta}\left[\int K\Big(q, \at_q , \at_q \circledast
 (T,x)\Big) \; 
 \mu_\theta(dx,d\ct,dq) \right]. 
\end{align*}
\end{proof}

It can be noticed that the map $q\mapsto \at_q$ is non-decreasing càdlàg
(backwards in time) and that we have, for $j\in \cup _{n\in \N} J^{(n)}$,
$x_j\in \at_{\theta_j}$: $ \at_{\theta_j-}= \at_{\theta_j}\circledast
(\ct^j,x_j)$. In particular, we can recover the random measure $\cn_\theta$ from
the jumps of the process $( \at_q, q\in \Theta^\psi\cap (-\infty ,\theta])$.
This and the natural compatibility relation of $\cn_\theta$ with respect to
$\theta$ gives the next Corollary.

\begin{cor}
 \label{cor:cn-T}
Let $\psi$ be a branching mechanism satisfying Assumptions 1 and
2. Let $ (\ct_\theta, \theta\in \Theta^\psi)$ be defined under $\N^\psi$. Let
\[
\cn=\sum_{j\in
 J}\delta_{(x_j,\ct^j,\theta_j)}
\]
 be the random point measure defined as follows:
\begin{itemize}
\item The set $\{\theta_j; j\in J\}$ is the set of jumping times of the
 process $ (\ct_\theta, \theta \in \Theta^\psi)$: for $j\in J$,
 $\ct_{\theta_j-}\neq \ct_{\theta_j}$.
\item The real tree $\ct^j$ is the closure of $\ct_{\theta_j-} \setminus
 \ct_{\theta_j}$.
\item The point $x_j$ is the root of $\ct^j$ (that is $x_j$ is the only
 element $y\in \ct_{\theta_j-}$ such that $x\in \ct^j$ implies
 $\llbracket y,x \rrbracket \subset \ct^j$).
\end{itemize}
Then the backward point process $\theta \mapsto \ind_{\{\theta\leq q'\}}
\cn(dx, d\ct, dq')$ defined on $\Theta^\psi$ has predictable
compensator: 
\[
\mu(dx,d\ct,dq)=\bm^{\ct_q}(dx) \bN^{\psi_q}[d\ct]\; \ind_{\{q\in
 \Theta^\psi\}}\; dq , 
\]
with respect to the backward left-continuous filtration $\cf=(\cf_\theta,
\theta\in \Theta^\psi)$ defined by:
\[
\cf_\theta=\sigma((x_j, \ct^j,\theta_j); \theta\leq \theta_j)=
\sigma(\ct_{q-}; \theta\leq q).
\]
More precisely, for any non-negative predictable process $K$ with respect to the
backward filtration $\cf$, we have:
\begin{equation}
 \label{eq:compensation}
 \N^\psi \left[\int \cn(dx, d\ct,dq)\;
 K\Big( q,\ct_{q}, \ct_{q-}\Big) \right] 
 =\N^\psi \left[\int \mu(dx,dT,dq) \; K\Big( q,\ct_q , \ct_q
 \circledast (T,x)\Big) \; \right]. 
\end{equation}
\end{cor} 

\begin{rem}
 \label{rem:Assumption1-cor}
Notice that Assumption 2 is assumed only for technical measurability condition,
see Remark \ref{rem:Assumption2}. We conjecture that this results holds also if
Assumption 2 is not in force. 
\end{rem}

As a consequence, thanks to property 3 of Theorem \ref{ADAvExplo}, we get, with
the convention $\sup\emptyset=\theta_\infty $, that:
\[
A=\sup\{\theta_j, j\in J \text{ and } \sigma^j=+\infty\}
\quad\text{with}\quad
\sigma_j=\bm^{\ct^j}(\ct^j).
\]

\subsection{Proof of Theorem \ref{theo:at=ct}}
\label{sec:proof-thm}
By construction, it is clear that the process $(\at_q, q\in \Theta^\psi\cap
(-\infty ,\theta])$ is a backward Markov process with respect to the backward
filtration $(\cf_q, q\in \Theta^\psi\cap (-\infty ,\theta])$. By construction
this process is càglàg in backward time. Since the process $(\ct_q, q\in
\Theta^\psi)$ is a forward càdlàg Markov process, it is enough to check that for
$\theta_0\in \Theta^\psi$, such that $\theta_0<\theta$, the two dimensional
marginals $ (\at_{\theta_0}, \at_{\theta})$ and $(\ct_{\theta_0},\ct_{\theta})$
have the same distribution.
\par Replacing $\psi$ by $\psi_{\theta_0}$, we can assume that $\theta_0=0$ and
$0<\theta$. We shall decompose the big tree $\ct_{0}$ conditionally on the small
tree $\ct_{\theta}$ by iteration. This decomposition is similar to the one which
appears in \cite{Abraham2007d} or \cite{Voisin2010} for the fragmentation of the
(sub)critical Lévy tree, but roughly speaking the fragmentation is here frozen
but for the fragment containing the root.
\par We set $\ct^{(0)}=\ct_\theta$ and $\tilde \bm^{(0)}=\bm^{\ct_\theta}$, so
that $(\at^{(0)}, \bm^{{(0)}})$ and $(\ct^{(0)},\tilde \bm^{(0)}) $ have the
same distribution. Recall notation $\cm^\uparrow$ from \reff{mfleche} as well as
\reff{eq:T-from-LT}: $\ct_0=\ct^{(0)} \circledast_{i\in I_\theta^{\uparrow,1}}
(\ct^i,x_i)$, where we write $ I_\theta^{\uparrow,1}= I_\theta^{\uparrow}$ and
where $\cp^1=\sum_{i\in I_\theta^{\uparrow,1}} \delta_{(x_i,\ct^i,\theta_i)}$
is, conditionally on $\ct^{(0)}$, a Poisson point measure with intensity:
\[
\nu^1 (dx,d\ct',dq)=\tilde \bm^{(0)} (dx) \Big( 2\beta
 \Ex^\psi[d\ct'] + \int_{(0,+\infty)} \Pi(dr) \; r \expp{-q r}
 \Pr^\psi_r(d\ct') \Big)\; \ind_{(0,\theta]}(q)\; dq.
\]

For $i\in I_\theta^{\uparrow,1}$, we define the sub-tree of $\ct^i$:
\[
\tilde \ct^i=\{x\in \ct^i; \cm^\uparrow(\rrbracket x_i,x
\llbracket\times [0,\theta_i])=0\}.
\]
Since $\ct^i$ is distributed according to $\N^{\psi}$ (or to $\Pr^\psi_{r_i}$
for some $r_i>0$), using the property of Poisson point measures, we have that
conditionally on $\ct^{0}$ and $\theta_i$, the tree $\tilde \ct^i$ is
distributed as $\Lambda_{\theta_i}(\ct,\cm)$ under $\N^\psi$ (or under
$\Pr^\psi_{r_i}$) that is the distribution of $\tilde \ct^i$ is
$\N^{\psi_{\theta_i}}[d \ct]$ (or $\Pr^{\psi_{\theta_i}} _{r_i}(d\ct)$), thanks
to the special Markov property. Furthermore we have $\ct^i=\tilde \ct^i
\circledast_{i'\in I_{\theta,i}^{\uparrow,2}} (\ct^{i'},x_{i'})$ where
\[
\sum_{i'\in I_{\theta,i}^{\uparrow,2}}
\delta_{(x_{i'},\ct^{i'},\theta_{i'})}
\]
 is, conditionally on $\ct^{(0)}$ and $\tilde \ct^i$ a Poisson point measure on
$\tilde \ct^i\times \EsAr \times (0,\theta]$ with intensity: 
\[
 \mathbf{m}^{\tilde \ct^i}(dx) \Big( 2\beta
 \Ex^\psi(d\ct') + \int_{(0,+\infty)} \Pi(dr) \; r \expp{-q r}
 \Pr^\psi_r(d\ct') \Big)\; \ind_{[0,\theta_i)}(q)\; dq.
\]
Thus we deduce, using again the special Markov property, that:
\[
\tilde \cn^1_{\theta}(dx,d\ct,dq)=\sum_{i\in I^{\uparrow,1}}
\delta_{(x_i,\tilde \ct^i,\theta_i)}(dx,d\ct,dq) 
\]
is conditionally on $\ct^{0}$ a Poisson point measure on $\ct^{(0)}\times
\T\times \Theta^\psi$ with intensity:
\[
\tilde \mu^1(dx,d\ct,dq)=
\tilde \bm^{(0)}_q (dx) \bN^{\psi_q}[d\ct]\; \ind_{[0,\theta)}(q)\; dq,
\]
with $\tilde \bm^{(0)}_q (dx)=\tilde \bm^{(0)} (dx)$. We set
$\ct^{(1)}=\ct^{(0)} \circledast_{i\in I_\theta^{\uparrow,1}} (\tilde
\ct^i,x_i)$ for the first generation tree and for $q\in [0,\theta]$:
\[
\tilde \bm^{(1)}_q (dx)=\sum _{i\in I_{\theta}^{\uparrow,1}}
 \bm^{\tilde \ct^i}(dx) \ind_{[0,\theta_i)}(q).
\]
See Figure \ref{fig:generation} for a simplified representation. We get that $
(\at^{(1)}_\theta, (\bm^{(1)}_q, q\in [0,\theta]), \at^{(0)}, \bm^{\at^{(0)}})$
and $ (\ct^{(1)},(\tilde \bm^{(1)}_ q, q\in [0,\theta]), \ct^{(0)},\tilde
\bm^{(0)})$ have the same distribution.

\begin{figure}[htbp]
 \begin{center}
 \includegraphics{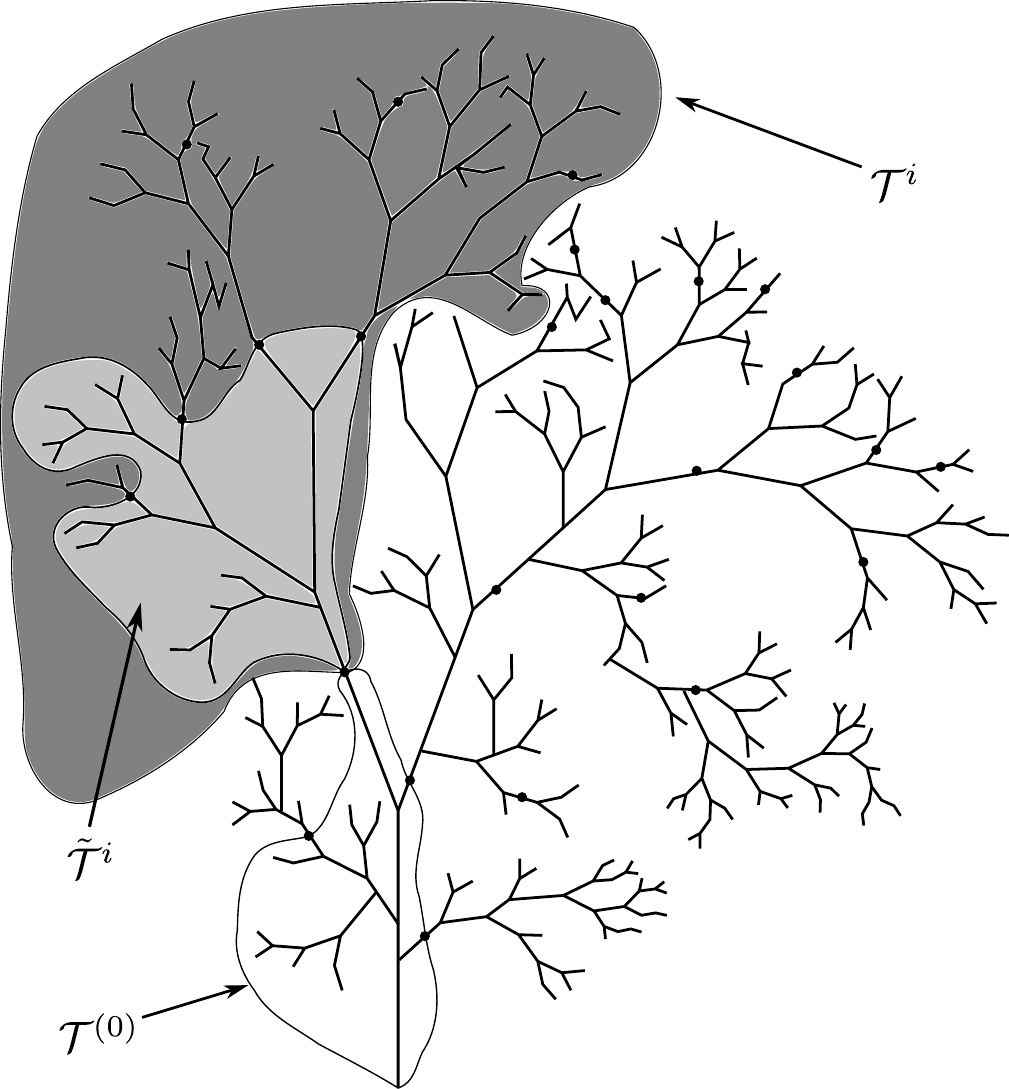}
 \caption{The tree $\ct_0$, $\ct^{(0)}$, and a tree $\ct^i$ and its
 sub-tree $\tilde \ct^i$ belonging to the first generation tree
 $\ct^{(1)}\setminus \ct^{(0)}$. }
\label{fig:generation}
 \end{center}
\end{figure}

Furthermore, by collecting all the trees grafted on $\ct^{(1)}$, we get that
$\ct=\ct^{(1)} \circledast_{i'\in I_{\theta}^{\uparrow,2}} (\ct^{i'},x_{i'})$
where $ I_{\theta}^{\uparrow,2}=\cup _{i\in I_{\theta}^{\uparrow,1} }
I_{\theta,i}^{\uparrow,2}$ and
\[
\cp^2=\sum_{i'\in I_{\theta}^{\uparrow,2}}
\delta_{(x_{i'},\ct^{i'},\theta_{i'})}
\]
is, conditionally on $ (\ct^{(1)},(\tilde \bm^{(1)}_ q, q\in [0,\theta]),
\ct^{(0)},\tilde \bm^{(0)})$ a Poisson point measure on $\ct^{(1)}\times \EsAr
\times (0,\theta]$ with intensity:
\[
\nu^2(dx,d\ct,dq)=\tilde \bm^{(1)}_q(dx)\; \Big( 2\beta \Ex^\psi(d\ct')
+ \int_{(0,+\infty)} \Pi(dr) \; r \expp{-q r} \Pr^\psi_r(d\ct')
\Big)\; \ind_{[0,\theta]}(q)\; dq.
\]

Notice that:
\begin{equation}
 \label{eq:ctm}
\ct^{(1)}=\{x\in \ct_0;\cm^\uparrow(\llbracket \emptyset, x
 \llbracket \times [0,\theta])\leq 1\}
\quad\text{and}\quad
\tilde \bm^{(1)}_\theta(dx)+
 \tilde \bm^{(0)}(dx)=\ind_{\ct^{(1)}}(x)\;\bm^{\ct_0} (dx). 
 \end{equation} 

 Then we can iterate this construction, and by taking increasing limits we
obtain that the pair $((\cup_{n\in \N} \at_\theta^{(n)}, \sum_{n\in \N}
\bm^{(n)}_\theta), \at_0)$ has the same distribution as $(\ct',\ct^{(0)})$,
where:
\[
\ct'=\{x\in \ct_0;\cm^\uparrow(\llbracket \emptyset, x
 \llbracket \times [0,\theta])<+\infty \}
\quad\text{and}\quad
\tilde \bm '(dx)=\ind_{\ct'}(x)\;\bm^{\ct_0}(dx) . 
\]
To conclude, we need to check first that the completion of $\ct'$ is $\ct_0$, or
as $\ct_0$ is complete that the closure of $\ct'$ as a subset of $\ct_0$ is
exactly $\ct_0$ and then that $\bm^{\ct_0}(\ct'^c)=0$.
\par Notice that $\cm^\uparrow $ has less marks than $\cm$. Then Proposition 1.2
in \cite{Abraham2007d} in the case when $\beta=0$ or an elementary adaptation of
it in the general framework of \cite{Voisin2010}, gives there is no loss of mass
in the fragmentation process. This implies that, if $\psi$ is (sub)critical,
then:
\begin{equation}
 \label{eq:bmminfini=0}
\bm^{ \ct_0}(\{x\in \ct_0; \cm
(\llbracket \emptyset, x \llbracket \times [0,\theta])=\infty
\}=0.
\end{equation}
Then, if $\psi$ is super-critical, by considering the restriction of $\ct_0$ up
to level $a$, $\pi_a(\ct_0)$, and using a Girsanov transformation from
Definition \ref{def:Girsanov} with $\theta=\theta^*$ and \reff{eq:bmminfini=0},
we deduce that \reff{eq:bmminfini=0} holds for $\pi_a(\ct_0)$. Since $a$ is
arbitrary, we deduce by monotone convergence that \reff{eq:bmminfini=0} holds
also in the super-critical case. Thus we have $\bm^{ \ct_0}(\ct'^c)=0$. Since
the closed support of $\bm^{\ct_0}$ is the set of leaves $\mathrm{Lf}(\ct_0)$,
we deduce that $\mathrm{Lf}(\ct')$ is dense in $\mathrm{Lf}(\ct_0)$ and, as
$\ct'$ and $\ct_0$ have the same root, that $\mathrm{Sk}(\ct')=
\mathrm{Sk}(\ct_0)$. This implies that the closure of $\ct'$ is $\ct_0$. This
ends the proof.

\section{Application to overshooting}
\label{sec:appli}
We assume that $\psi$ is critical, $\theta_\infty <0$ and Assumptions 1 and 2
hold. We shall write $u^\theta$ (resp. $b^\theta$) for the solution of
\reff{eq:def-u} (resp. \reff{bh}) when $\psi$ is replaced by $\psi_\theta$, for
$a\geq 0$, $h>0$ and $t\in [0,h)$:
\begin{equation}
 \label{eq:ub}
 \int_{u^\theta(a,\lambda)}^\lambda \frac{dr}{\psi_\theta(r)} = a, 
\quad \text{and}\quad
 b^\theta_h(t)=b^\theta(h-t)
\quad\text{with}\quad
 \int_{b^\theta(h)}^\infty \frac{dr}{\psi_\theta(r)} = h.
\end{equation}
We have $u^\theta (a,b^\theta(h-a))=b^\theta(h)$. Notice that $\partial_h
b^\theta(h)/ \psi_\theta(b^\theta(h))=-1$ and that $\partial_\lambda
u^\theta(a,\lambda)=\psi_\theta(u^\theta(a,\lambda))/\psi_\theta(\lambda)$ which
implies that:
\begin{equation}
 \label{eq:dub}
\partial_\lambda
u^\theta(a,b^\theta(h-a))=
\frac{\psi_\theta(b^\theta(h))}{\psi_\theta(b^\theta(h-a))} =
-\frac{\psi_\theta(b^\theta(h))}{\psi_\theta(b^\theta(h-a))^2} \partial_h
b^\theta(h-a). 
\end{equation}

We set for $\theta\in \Theta^\psi$ and
$\lambda\geq 0$:
\begin{equation}
 \label{eq:gamma}
 \gamma_\theta(\lambda) = 
\psi_\theta'(\lambda) - \psi_\theta'(0)=\psi'(\lambda+\theta)
-\psi'(\theta)=\partial_\theta \psi_\theta(\lambda). 
\end{equation}
Notice the function $\gamma_\theta$ is non-negative and non-decreasing. 
\par Recall that $\bar \theta=\psi^{-1}\circ \psi(\theta)$. We deduce from
\reff{eq:ub} that for $\theta\in \Theta^\psi$, $\theta<0$ and $h>0$:
\begin{equation}
 \label{eq:b=b}
\bar \theta + b^{\bar \theta}(h)=
\theta + b^{\theta}(h)
\quad\text{and}\quad
\psi_{\bar \theta} (b^{\bar \theta}(h))=
\psi_{\theta} (b^{\theta}(h)).
\end{equation}

\subsection{Exit times}\label{ExitTimes}

Let $h>0$. We are interested in the first time when the process of growing trees
exceeds height $h$, in the following sense. 

\begin{definition} The first exit time out of $h$ is the (possibly
infinite) number $A_h$ defined by
\[ 
A_h = \sup \{ \theta \in \Theta^\psi,\ H_{max}(\ct_\theta) > h\}, 
\]
with the convention that $\sup\emptyset=\theta_\infty $. 
\end{definition}

The constraint not to be higher than $h$ will be coded by the function
$b^\theta(h)$ which is the probability (under $\N^{\psi}$) for the tree $
\ct^\theta$ of having maximal height larger than $h$. By definition of the
function $b$, we have for $\theta\in \Theta^\psi$:
\begin{equation}
 \label{eq:b-Qh}
\N^\psi[\theta\leq A_h]=\N^\psi\left[H_{max}(\ct_\theta) \geq h
\right]=b^\theta(h).
\end{equation}

\begin{prop}
 \label{prop:dbq}
 Let $\psi$ be a critical branching mechanism with $\theta_\infty <0$
 and satisfying Assumptions 1 and 2. The function $\theta\mapsto
 b^\theta_h$ is of class $\cc^1$ on $(\theta_\infty , +\infty )$. And,
 under $ \N^\psi$, the distribution of $A_h$ on $(\theta_\infty ,+\infty
 )$ has density $\theta\mapsto -\partial_\theta b^\theta(h)$ with
 respect to the Lebesgue measure. We also have the following expression
 for the density of $A_h$ on $(\theta_\infty ,+\infty )$. Let
 $\theta_\infty <\theta$ and $h>0$. Then:
\[
-\partial_\theta b^\theta(h) 
= \psi_\theta(b^\theta(h))\int_0^h da \;
\frac{\gamma_\theta(b^\theta(a)) }{\psi_\theta(b^\theta(a)) }
= \int_0^{h }da \; \gamma_{ \theta} (b^{ \theta}(h-a)) \;
\expp{-\psi'({ \theta}) a - \int_0^a dx\;
\gamma_{ \theta} (b^{ \theta}(h-x)) }.
\]
\end{prop}
Notice that the distribution of $A_h$ might have an atom at
$\theta_\infty $.
\begin{proof}
 Notice that for $\theta_\infty <\theta$, we have $\lim_ {\lambda
\rightarrow+\infty } \psi''(\lambda)=\beta$ and $\lim_ {\lambda
\rightarrow+\infty } \psi'(\lambda)=+\infty $. In particular
$\psi_\theta'(\lambda)/\psi_\theta(\lambda)$ is bounded for $\lambda$ large
enough. This implies that $\int ^{+\infty }dr\; \psi'_\theta(r)/\psi_\theta(r)^2
$ is finite thanks to Assumption 2. We deduce that the function $\theta\mapsto
b^\theta_h$ is of class $\cc^1$ on $(\theta_\infty , +\infty )$ and, thanks to
\reff{eq:b-Qh}, that under $ \N^\psi$, the distribution of $A_h$ on
$(\theta_\infty ,+\infty )$ has density $\theta\mapsto -\partial_\theta
b^\theta(h)$ with respect to the Lebesgue measure.
\par Taking the derivative with respect to $\theta$ in the last term of
\reff{eq:ub}, using \reff{eq:gamma} and the change of variable $r=b^\theta(a)$
gives the first equality of the Proposition:
\begin{equation}
 \label{eq:deriv-bh}
-\partial_\theta b^\theta(h) = \psi_\theta(b^\theta(h))
\int_{b^\theta(h)}^{+\infty } dr
\frac{\gamma_\theta(r)}{\psi_\theta(r)^2}
= \psi_\theta(b^\theta(h))\int_0^h da \;
\frac{\gamma_\theta(b^\theta(a)) }{\psi_\theta(b^\theta(a)) } \cdot
\end{equation}
From \reff{eq:ub} we get that $ \partial_t b_h^\theta(t)
={\psi_\theta(b_h^\theta(t))}$. Hence, we have:
\[
\int_0^t \psi'_\theta(b^\theta_h(r))\; dr= \int_0^t
\frac{\psi_\theta'(b_h^\theta(r))}{\psi_\theta(b_h^\theta(r))}
\partial_t b_h^\theta(r) \;dr= \log\left(
 \frac{\psi_\theta(b_h^\theta(t))}{\psi_\theta(b_h^\theta(0))} 
\right). 
\]
This gives:
\begin{equation}
 \label{eq:int-psi'b}
\int_0^t \gamma_\theta(b_h^\theta(r)) dr 
= \int_0^t \psi'_\theta(b^\theta_h(r))\; dr - t \psi'(\theta) 
 = \log\left( \frac{\psi_\theta(b_h^\theta(t))}{\psi_\theta(b_h^\theta(0))}
\right) -t\psi'(\theta) .
\end{equation}
We deduce that:
\[
\int_0^{h }da \; \gamma_{ \theta} (b^{ \theta}(h-a)) \;
\expp{-\psi'({ \theta}) a - \int_0^a dx\;
\gamma_{ \theta} (b^{ \theta}(h-x)) }
= \psi_\theta(b^\theta(h))\int_0^h da \;
\frac{\gamma_\theta(b^\theta(a)) }{\psi_\theta(b^\theta(a)) }\cdot
\]
This proves the second equality of the Proposition. 
\end{proof}

Since we will also be dealing with super-critical trees, there is always the
positive probability that in the Poisson process of trees an infinite tree
arises, which will be grafted onto the process, effectively making it infinite
and thus outgrowing height $h$. In the next proposition, we will compute the
conditional distribution of overshooting time $A_h$, given $A$. Note that
we always have $A\leq A_h$. 
\begin{proposition} 
\label{prop:Qh-A}
Let $\psi$ be a critical branching mechanism with $\theta_\infty <0$
 and satisfying Assumptions 1 and 2. For $\theta_\infty <\theta_0 <\theta$ and
$\theta_0<0$ (that is $\psi_{\theta_0}$ super-critical), we have, with
$\hat{\theta} = \bar{\theta}_0-\theta_0+\theta$:
\begin{align*}
\Ex^\psi[A_h \ge \theta | A=\theta_0 ] 
&= 1 -
\psi'(\hat\theta)
\psi_{\hat\theta}(b^{\hat\theta}(h))\int_{b^{\hat\theta}(h)}^{+\infty }
\frac{dr}{\psi_{\hat\theta} (r)^2} ,\\
\Ex^\psi[A_h =A| A=\theta_0 ] 
&= \psi'(\bar \theta_0)
\psi_{\bar \theta_0}(b^{\bar \theta_0}(h))\int_{b^{\bar \theta_0}(h)}^{+\infty }
\frac{dr}{\psi_{\bar \theta_0} (r)^2}\cdot
\end{align*}
\end{proposition}
Since $\psi_{\bar \theta_0}$ is sub-critical, we have $\psi'(\bar \theta_0)>0$
and $\psi_{\bar \theta_0}(r) \sim r \psi'(\bar \theta_0)$ when $r$ goes down to
$0$. Since $\lim_{h\rightarrow +\infty } b^{\bar \theta_0}(h)=0$, we deduce
that:
\[
\lim_{h\rightarrow+\infty } \Ex^\psi[A_h =A|
A=\theta_0 ]=1.
\] 
This has a straightforward explanation. If $h$ is very large, with high
probability the process up to $A$ will not have crossed height $h$, so that the
first jump to cross height $h$ will correspond to the grafting time of the first
infinite tree which happens at the ascension time $A$. 
\par We also deduce from \reff{eq:b=b} that:
\begin{equation}
 \label{eq:Q=A|A}
\Ex^\psi[A_h =A| A=\theta_0 ] = \psi'(\bar \theta_0)
\psi_{\theta_0}(b^{\theta_0}(h))\int_{b^{\theta_0}(h)}^{+\infty }
\frac{dr}{\psi_{ \theta_0} (r)^2}\cdot
\end{equation}

\begin{proof} 
We use the notation $\cz^\theta_h=\cz_h(\ct^\theta)$ and $\cz_h=\cz_h(\ct^0)$.
We have:
\begin{align*}
\Ex^\psi[ A_h \ge \theta | A=
\theta_0 ] 
= 
\Ex^\psi [ \cz_{h}^\theta >0 | A=\theta_0 ] 
 & = \Ex^\psi [ \cz_h^{A+(\theta-\theta_0)} >0
 | A=\theta_0 ] \\
 & = \psi'(\bar{\theta}_0) \Ex^\psi \left[ \sigma_0 \ind_{\{
\cz_h^{(\theta-\theta_0)} >0\} }
\expp{-\psi(\theta_0)\sigma_0} \right] \\
 & = \psi'(\bar{\theta}_0) \Ex^{\psi_{\bar{\theta}_0}} \left[ \sigma_0
\ind_{\{ \cz_h^{(\theta-\theta_0)} >0\} } \right]\\
 & = \psi'(\bar{\theta}_0) \Ex^\psi \left[ \sigma_{\bar{\theta}_0} \ind_{
\{\cz_{h}^{\bar{\theta}_0+(\theta-\theta_0)} >0\}} 
\right]\\
&=\psi'(\bar{\theta}_0) \Ex^\psi \left[ \sigma_{\bar{\theta}_0} \ind_{
\{\cz_{h}^{\hat \theta} >0\}} 
\right],
\end{align*}
where we used (2) of Theorem \ref{ADAvExplo} for the third equality, Girsanov
formula \reff{Girsanov} for the fourth and the homogeneity property of Theorem
\ref{ProcArbre*} in the fifth. We now condition with respect to
$\ct^{\hat\theta}$. The indicator function being measurable, the only quantity
left to compute is the conditional expectation of $\sigma_{\bar{\theta}_0}$
given $\ct^{\hat\theta}$. Thanks to Lemma \ref{lem:CroiSig}, the fact that
$\hat\theta>0$ and the homogeneity property, we get:
\[
 \Ex^\psi[ A_h \ge \theta | A=
\theta_0 ] 
 = \psi'(\hat\theta)\Ex^\psi\left[ \sigma_{\hat\theta}
\ind_{\{\cz_h^{\hat\theta} >0 \} } \right]
=\psi'(\hat\theta)\Ex^{\psi_{\hat\theta}} \left[\sigma 
\ind_{\{\cz_{h} > 0 \}} \right]. 
\]
Using that $\Ex^{\psi_{\hat\theta}}[\sigma]=1/\psi'(\hat\theta)$, which can be
deduced from \reff{eq:Ns}, we get:
\begin{align*}
 \Ex^\psi[ A_h \ge \theta | A=
\theta_0 ] 
& =\psi'(\hat\theta) \Ex^{\psi_{\hat\theta}}[\sigma] - \psi'(\hat\theta)
\Ex^{\psi_{\hat\theta}}
\Big[\int_0^h \cz_a da \ind_{\{\cz_h =0 \}} \Big] \\ 
& = 1 - \psi'(\hat\theta) \int_0^h da\; \lim_{\lambda \rightarrow
\infty} \Ex^{\psi_{\hat\theta}} \Big[ \cz_a 
\expp{-\lambda\cz_h} \Big]. 
\end{align*}
Now, conditioning by $\cz_a$ and using $\lim_{\lambda \rightarrow \infty}
u^{\hat\theta}(h-t,\lambda)=b_h^{\hat\theta}(t)$ as well as \reff{eq:NZ}, we
get:
\[ 
 \lim_{\lambda \rightarrow
\infty} \Ex^{\psi_{\hat\theta}} \Big[\cz_a 
\expp{-\lambda \cz_h} \Big] = \lim_{\lambda \rightarrow
\infty} \Ex^{\psi_{\hat\theta}}
\Big[
\cz_a \expp{-\cz_a u^{\hat\theta}(h-a,\lambda)}\Big]
= \Ex^{\psi_{\hat\theta}}[ \cz_a
\expp{-\cz_a
b_h^{\hat\theta}(a)} ]= 
\partial _\lambda u^{\hat \theta} (s,b_h^{\hat
 \theta}(a)).
\] 
Then use \reff{eq:dub} to get:
\begin{align*}
\int_0^h da\; \lim_{\lambda \rightarrow
\infty} \Ex^{\psi_{\hat\theta}} \Big[ \cz_a 
\expp{-\lambda\cz_h} \Big]
=\int_0^h da\ 
\partial _\lambda u^{\hat \theta} (s,b_h^{\hat
 \theta}(a))
&=\psi_{\hat \theta} (b^{\hat\theta}(h)) \int_0^h da\ \frac{|\partial_h
b^{\hat \theta}(h-a)|}{\psi_{\hat\theta}(b^{\hat \theta}(h-a))^2} \\
&= \psi_{\hat\theta}(b^{\hat\theta}(h))\int_{b^{\hat\theta}(h)}^{+\infty }
\frac{dr}{\psi_{\hat\theta} (r)^2},
\end{align*}
and thus deduce the first equality of the Proposition. Notice $\int^{+\infty }
dr/\psi_\theta(r)^2 <+\infty $ thanks to Assumption 2 (in fact this is true in
general). Let $\theta$ go down to $\theta_0$ and use the fact that
$\N^\psi$-a.e. $A\leq A_h$ to get the second equality.
 \end{proof} 

\begin{remark} In the quadratic case $\psi(u)=\beta u^2$, we can obtain closed
formulae. For all $\theta>0$, we have: 
\[ 
u^\theta (t,\lambda) = \frac{2\theta \lambda}{(2\theta+\lambda) \exp(2\beta
\theta t) -\lambda}
\quad\text{and}\quad 
b^\theta(t)=\frac{2\theta}{\expp{2\beta \theta t}-1}\cdot
\]
We have the following exact expression of the conditional distribution for
$\theta_0< \theta$, $\theta_0<0$ and with $\bar
\theta_0=\val{\theta_0}=-\theta_0$ and $\hat\theta=\theta+2\val{\theta_0}$: 
\begin{align*}
\Ex^\psi[A_h \ge \theta | A=\theta_0 ] 
&= 1+ (\beta\hat{\theta}
h)/ \mathrm{sinh}^2(\beta\hat{\theta} h) -
\mathrm{cotanh}(\beta\hat{\theta} h),\\
\Ex^\psi[A_h =A| A=\theta_0 ] 
&= \beta\theta_0
h/ \mathrm{sinh}^2(\beta{\theta_0}
h)-\mathrm{cotanh}(\beta{\theta_0} h) .
\end{align*}
Notice that $\lim_{\theta_0\rightarrow-\infty }\Ex^\psi[A_h =A| A=\theta_0
]=1$. This correspond to the fact that if $A$ is large, then the tree $\ct_A$ is
small and has little chance to cross level $h$. (Notice that $\ct_A$ has finite
height but $\ct_{A-}$ has infinite height.) Thus the time $A_h$ is equal to
the time when an infinite tree is grafted, that is to the ascension time $A$.
\end{remark}

\subsection{Distribution of the tree at the exit time $A_h$}

Before stating the theorem describing the tree before it overshoots a given
height $h>0$ under the form of a spinal decomposition, we shall explain how this
spine is distributed. Recall \reff{eq:gamma} for the definition of
$\gamma_\theta$.

\begin{lemma} 
\label{lem:xi}
Let $\psi$ be a critical branching mechanism satisfying Assumptions 1
and 2. Let $\theta\in \Theta^\psi$. The non-negative function
\begin{equation} 
\label{**}
f:t\mapsto \gamma_\theta(b_h^\theta(t))\exp\Big( -\int_0^t
\gamma_\theta(b_h^\theta(r)) dr 
\Big) 
\end{equation}
is a probability density on $[0,h)$ with respect to Lebesgue measure. If $\xi$
is a random variable whose distribution is $f$, then we have
$\E[\exp(-\psi'(\theta) \xi)]<+\infty $.
\end{lemma}
Notice the integrability property on $\xi$ is trivial if $\theta\geq 0$. 
\begin{proof} Notice that $f=g'\expp{-g}$ with $g(t)=\int_0^t \gamma_\theta (
b_h^\theta(r))\; dr$. Thus we have $\int_0^h f= \int_0^h g' \expp{-g}
=\expp{-g(0)}-\expp{-g(h)}$ and $f$ is a density if and only if $g(h)=\infty$.
We deduce from \reff{eq:int-psi'b} that $\int_0^t \gamma_\theta(b_h^\theta(r))
dr $ diverges as $t$ goes to $h$. The last part of Proposition \ref{prop:dbq}
implies that $\expp{-\psi'(\theta) \xi}$ is integrable.
\end{proof}

Recall Equation \reff{def:gref} defining the grafting procedure. 

\begin{theorem} \label{ArbreAvantH} 
Let $\psi$ be a critical branching mechanism satisfying Assumptions
1 and 2. Let $\theta_\infty<\theta$ and let $F$ be a non-negative
measurable functional 
on $\EsAr^2$. Then, we have:
 \begin{multline*} 
 \Ex^\psi\left[F(\ct_{A_h}\; ;\; \ct_{A_h-}) | A_h=\theta
 \right] \\ 
= \inv{\mathbf{E}\left[\expp{- \psi'(\theta)
 H_\mathbf{x} } \right]}
\mathbf{E}\left[F\Big(\llbracket \emptyset,\mathbf{x} \rrbracket
\circledast_{i\in I}
(\ct^i,x_i)\; ; \;(\llbracket \emptyset,\mathbf{x} \rrbracket \circledast_{i\in
I}
(\ct^i,x_i)) \circledast (T,\mathbf{x})\Big) \;\expp{- \psi'(\theta)
 H_\mathbf{x} } \right],
 \end{multline*}
where the spine $\llbracket \emptyset,\mathbf{x} \rrbracket$ is identified with
the interval $[0,H_\mathbf{x}]$ (and thus $y\in \llbracket \emptyset,\mathbf{x}
\rrbracket$ is identified with $H_y$) and:
\begin{itemize} 
 \item The random variable 
 $H_{\mathbf{x}}$ is distributed with density given by \reff{**}. 
 \item Conditionally on $H_\mathbf{x} $, sub-trees are grafted on
 the spine $[0,H_\mathbf{x}]$ according to a Poisson point
 measure $\cn=\sum_{i\in I} \delta_{(x_i,\ct^i)}$ on
 $[0, H_\mathbf{x}]\times \EsAr$ with
 intensity:
\begin{multline} 
\label{eq:def-nu}
\nu_\theta(da,d\ct)=da \ \Big( 2 \beta (\theta+b_h^\theta(a))
\Ex^{\psi_{\theta}}[d\ct,
H_{max}(\ct) < h-a ] \\
+ \int_{(0,+\infty)} r \Pi_{\theta+b_h^{\theta}(x)}(dr)
\Pr_r^{\psi_{\theta}}(d\ct, H_{max}(\ct) < h-a ) \Big) .
\end{multline}
 \item Conditionally on $H_\mathbf{x}
 $ and on $\cn$, $T$ is a random variable on $\EsAr$ with
distribution \[ \bN^{\psi_\theta}[dT|H_{max}(T) > h -H_\mathbf{x}]. \]
 \end{itemize}
\end{theorem}
 
In other words, conditionally on $\{A_h = \theta \}$, we can describe the
tree before overshooting height $h$ by a spinal decomposition along the
ancestral branch of the point at which the overshooting sub-tree is grafted.
Conditionally on the height of this point, the overshooting tree has
distribution $\mathbf{N}^{\psi_\theta}[dT]$, conditioned on overshooting.
\par If $\theta>0$ then $\psi'(\theta)>0$, and we can understand the weight
$\expp{- \psi'(\theta) H_\mathbf{x} } /\mathbf{E}\left[\expp{- \psi'(\theta)
H_\mathbf{x} } \right]$ as a conditioning of the random variable $H_\mathbf{x}$
to be larger than an independent exponential random variable with parameter
$\psi'(\theta)$. 

\begin{remark} When $h$ goes to infinity, we have, for $\theta\geq 0$,
$\lim_{h\rightarrow +\infty } b^\theta(h)=0$ and thus the distribution of
$A_h$ concentrates on $\Theta^\psi\cap (-\infty ,0)$. For $\theta<0$ and
$\theta\in \Theta^\psi$, we deduce from \reff{eq:b=b} that $\lim_{h\rightarrow
+\infty } b^\theta(h)=\bar \theta -\theta>0$. And the distribution of $\xi$ in
Lemma \ref{lem:xi} clearly converges to the exponential distribution with
parameter $\gamma_\theta(b^\theta(+\infty ))=\psi'(\bar \theta) -\psi'(\theta)$.
Then the weight $\expp{- \psi'(\theta) H_\mathbf{x} } /\mathbf{E}\left[\expp{-
\psi'(\theta) H_\mathbf{x} } \right]$ changes this distribution. In the end,
$H_\mathbf{x} $ is asymptotically distributed as an exponential random variable
with parameter $\psi'(\bar \theta)$. Notice this is exactly the distribution of
the height of a random leaf taken in $\ct_A$, conditionally on $\{A=\theta\}$,
see Lemma 7.6 in \cite{Abraham2010}.
\end{remark}

\begin{remark}\label{rem:G} A direct application of Theorem \ref{ArbreAvantH}
with $F(\ct;\ct')$ chosen equal to
\begin{equation}
 \label{eq:GTT}
G(\ct;\ct')=\ind_{\{\bm^{\ct}(\ct)<+\infty,
 \bm^{\ct'}(\ct')=+\infty\}},
\end{equation}
allows to compute for $\theta<0$:
\[
\N^\psi[A=A_h|A_h=\theta]
=\left(\psi'(\bar
 \theta)- \psi'(\theta) \right) 
\frac{C(\theta,h)}{\psi'(\bar
 \theta) - \psi'(\theta) C(\theta,h)},
\]
where $C(\theta,h)=\psi'(\bar \theta)\psi_\theta(b^\theta(h) )
\int_{b^\theta(h)}^{+\infty } dr\; \psi_\theta(r)^{-2}
=\N^\psi[A=A_h|A=\theta]$. The last equality is a consequence of
\reff{eq:Q=A|A}. As $\lim_{h\rightarrow+\infty }
\N^\psi[A=A_h|A=\theta]=1$, we get that 
\[\lim_{h\rightarrow+\infty }\N^\psi[A=A_h|A_h=\theta]=1. \]
\end{remark}

\begin{remark}
 By considering the function $G$ in \reff{eq:GTT} instead of $F$ in the proof of
Theorem \ref{ArbreAvantH}, we can recover the distribution of $\ct_A$ given in
\cite{Abraham2010}, but we also can get the joint distribution of
$(\ct_{A-},\ct_A)$. Roughly speaking (and unsurprisingly), conditionally on
$\{A=\theta\}$, $\ct_{A-}$ is obtained from $\ct_A$ by grafting an independent
random tree $T$ on a independent leaf $x$ chosen according to $\bm^{\ct_A}(dx)$
and the distribution of $T$ is $\bN^{\psi_\theta}[dT,\;H_{max}(T)=+\infty ]$.
Notice that choosing a leaf at random on $\ct_A$ gives that the distribution of
$\ct_A$ is a size-biased distribution of $\N^{\psi_\theta}[d\ct]$. 
\end{remark}

\begin{proof}[Proof of Theorem \ref{ArbreAvantH}]
 Thanks to the compensation formula \reff{eq:compensation}, we can write, if $g$
is any measurable functional $\R\mapsto\R_+$ with support in $(\theta_\infty ,
+\infty )$:
\begin{multline*} 
\Ex^{\psi}[ F(\ct_{A_h}\; ;\; \ct_{A_h-})
g(A_h) ] \\ 
\begin{aligned}
& = \Ex^\psi\left[ \sum_{j\in J} \ind_{\{
H_{max}(\ct_{\theta_j}) <h \}}F(\ct_{\theta_j}\; ;\; \ct_{\theta_j} \circledast
(\ct^j,x_j)) g(\theta_j)
\ind_{\{ H_{x_j}+H_{max}(\ct^j) >h \} } \right] \\
& = \int_{\Theta^\psi} d\theta\ g(\theta) B(\theta,h),
\end{aligned}
\end{multline*}
where, using the homogeneity property and the Girsanov transformation
\reff{GirThetaBar}:
\begin{align*}
B(\theta,h) 
&= \Ex^\psi\left[
\ind_{\{H_{max}(\ct_\theta) < h\}} \int
\mathbf{m}^{\ct_\theta}(dx) \int \bN^{\psi_\theta}[dT] F(\ct_\theta\;
;\; \ct_\theta 
\circledast (T,x)) 
\ind_{\{H_x + H_{max}(T) >h\}} \right]\\
&= \Ex^{\psi_\theta}\left[
\ind_{\{H_{max}(\ct) < h\}} \int
\mathbf{m}^{\ct}(dx) \int \bN^{\psi_\theta}[dT] F(\ct\;;\; \ct
\circledast (T,x)) 
\ind_{\{H_x + H_{max}(T) >h\}} \right]\\
&= \Ex^{\psi_{\bar \theta}}\left[
\ind_{\{H_{max}(\ct) < h\}} \int
\mathbf{m}^{\ct}(dx) \int \bN^{\psi_\theta}[dT] F(\ct\;;\;\ct
\circledast (T,x)) 
\ind_{\{H_x + H_{max}(T) >h\}} \right].
\end{align*}
Notice we replaced only $\Ex^{\psi_{\theta}}$ by $\Ex^{\psi_{\bar \theta}}$ in
the last equality.
\par We explain how the term $\ind_{\{H_{max}(\ct) < h\}}$ change the
decomposition of $\ct$ according to the spine given in Theorem
\ref{theo:T/feuille}. Let $\Phi$ a non-negative measurable function defined on
$[0,+\infty )\times \T$ and $\varphi$ a non-negative measurable function defined
on $[0,+\infty )$. Using Theorem \ref{theo:T/feuille} and notations therein, we
get:
\begin{multline*}
 \Ex^{\psi_{\bar \theta}}
\left[ \int
\mathbf{m}^{\ct}(dx) \varphi(H_x) \expp{-\langle \cm_x,
\Phi \rangle } \ind_{\{H_{max}(\ct) < h\}} \right]\\
\begin{aligned}
&=\int_0^{\infty }da \; \varphi(a) \expp{-\psi'_{\bar \theta} (0) a }\;
\E\left[\expp{-\sum_{i\in I} \ind_{\{z_i\leq a\}} \Phi(z_i, 
\bar \ct^i)} \prod_{i\in I, z_i\leq a} \ind_{\{H_{max}( \bar \ct^i)
< h-z_i\}} \right] \\
&=\int_0^{h }da \; \varphi(a) \;
\exp\left(-\psi'({\bar \theta} ) a- \int_0^a dx\; \bN^{\psi_{\bar
 \theta}}\left[1- \expp{-\Phi(x, \ct)} 
 \ind_{\{H_{max}(\ct) < h-x\}}\right]\right).
\end{aligned}
\end{multline*}
Using the definition of $\bN^{\psi_{\bar \theta}}$, see \reff{def:NN},
\reff{eq:gamma} 
and the Girsanov transformation \reff{GirThetaBar}, we get:
\begin{multline*}
 \bN^{\psi_{\bar \theta}}\left[1-
 \expp{-\Phi(x, \ct)} \ind_{\{H_{max}(\ct) < h-x\}}\right]\\
\begin{aligned}
&=\gamma_{\bar \theta}\left(\N ^{\psi_{\bar \theta}}\left[1- \expp{-\Phi(x,
\ct)}
 \ind_{\{H_{max}(\ct) < h-x\}}\right] \right) \\
&=\gamma_{\bar \theta}\left(b^{\bar \theta}(h-x) + \N ^{\psi_{
 \theta}}\left[\left(1- \expp{-\Phi(x, \ct)}\right) 
 \ind_{\{H_{max}(\ct) < h-x\}}\right] \right). 
\end{aligned} 
\end{multline*}
Thanks to \reff{eq:gamma} and \reff{eq:b=b}, we have for $\lambda\geq
0$:
\[
\gamma_{\bar \theta}(b^{\bar \theta}(h-x) +\lambda)=
\gamma_{\theta+b^{\theta}(h-x) }(\lambda) + 
\gamma_{\theta}(b^{\theta}(h-x)) + \psi'(\theta)-\psi'(\bar \theta) .
\]
Take $\lambda= \N ^{\psi_{
 \theta}}\left[\left(1- \expp{-\Phi(x, \ct)}\right) 
 \ind_{\{H_{max}(\ct) < h-x\}}\right]$, 
to deduce that:
\begin{multline*}
 \Ex^{\psi_{\bar \theta}}
\left[ \int
\mathbf{m}^{\ct}(dx) \varphi(H_x) \expp{-\langle \cm_x,
\Phi \rangle } \ind_{\{H_{max}(\ct) < h\}} \right]\\
\begin{aligned}
&=\int_0^{h }da \; \varphi(a) \;
\exp\left(-\psi'({ \theta}) a - \int_0^a dx\;
\gamma_{\theta} (b^{ \theta}(h-x)) \right)\\
&\hspace{2cm}\exp\left(- \int_0^a dx\; \gamma_{\theta+ b^{
 \theta}(h-x)} 
 \left(\N^{\psi_{ \theta }}
\left[\left(1- \expp{-\Phi(x, \ct)} \right)
 \ind_{\{H_{max}(\ct) < h-x\}}\right]\right)\right)\\
&=\int_0^{h }da \; \varphi(a) \;
\exp\left(-\psi'({ \theta}) a - \int_0^a dx\;
\gamma_{ \theta} (b^{\theta}(h-x)) \right)
\E\left[\expp{-\sum_{i\in I} \ind_{\{ z_i\leq a\}} \Phi(z_i, \tilde
 \ct^i)} \right] ,
\end{aligned}
\end{multline*}
where under $\E$, $\sum_{i\in I} \delta_{(z_i, \tilde \ct^i)}(dz, d\ct)$ is a
Poisson point measure on $[0,h]\times \T$ with intensity $\nu_\theta$ in
\reff{eq:def-nu}. Since Laplace transforms characterize random measure
distributions, we get that for any non-negative measurable function $\tilde F$,
we have:
\begin{multline*}
 \Ex^{\psi_{\bar \theta}}
\left[ \int
\mathbf{m}^{\ct}(dx) \tilde F(H_x, \cm_x) \ind_{\{H_{max}(\ct) < h\}} 
\right]\\
= \int_0^{h }da \; 
\expp{-\psi'({\theta}) a - \int_0^a dx\;
\gamma_{\theta} (b^{\theta}(h-x)) } 
\E\left[\tilde F\left(a, \sum_{i\in I} \ind_{\{ z_i\leq a\}}
 \delta_{(z_i, \tilde \ct^i)}\right) \right] .
\end{multline*}
If we identify the spine $\llbracket \emptyset, x \rrbracket$ (with its metric)
to the interval $[ 0, H_x ]$ (with the Euclidean metric), we can use this result
to compute $B(\theta,h) $ with:
\[
\tilde F(H_x, \cm_x)=
\int \bN^{\psi_\theta}[dT\;|\;H_x + H_{max}(T) >h ] F(\ct\;;\; \ct
\circledast (T,x)), 
\]
$\cm_x=\sum_{i\in I_x}\delta_{(H_{x_i}, \ct^i)}$ and $\ct= [0, H_x ]
\circledast_{i\in I_x}(\ct^i,H_{x_i}) $. 
Since $\bN^{\psi_\theta}[ H_{max}(\ct) >h ]=
\gamma_\theta(b^\theta(h))$, we have:
\[
\gamma_\theta(b^\theta(h-H_x)) \tilde F(H_x , \cm_x)=
\int \bN^{\psi_\theta}[dT ] F(\ct\;;\; \ct \circledast (T,x))
\ind_{\{H_x + H_{max}(T) >h\}} .
\]
Therefore, we have:
\begin{align*}
B(\theta,h) 
&= \Ex^{\psi_{\bar \theta}}\left[
\ind_{\{H_{max}(\ct) < h\}} \int
\mathbf{m}^{\ct}(dx) \int \bN^{\psi_\theta}[dT] F(\ct\;;\;\ct
\circledast (T,x)) 
\ind_{\{H_x + H_{max}(T) >h\}} \right]\\
&=\int_0^{h }da \; \gamma_\theta(b^\theta(h-a)) 
\expp{-\psi'({ \theta}) a - \int_0^a dx\;
\gamma_{ \theta} (b^{ \theta}(h-x)) } 
\E\left[\tilde F\left(a, \sum_{i\in I} \ind_{\{ z_i\leq a\}}
 \delta_{(z_i, \tilde \ct^i)} \right) \right] .
\end{align*}
Thus, we get:
\begin{multline*}
\Ex^{\psi}[ F(\ct_{A_h}\; ;\; \ct_{A_h-})
g(A_h) ] \\
=\int_{\Theta^\psi} d\theta\ g(\theta) \int_0^{h }da \;
\gamma_\theta(b^\theta(h-a) ) \expp{-\psi'({ \theta}) a - \int_0^a dx\;
\gamma_{ \theta} (b^{ \theta}(h-x)) } \\
\E\left[\tilde F\left(a, \sum_{i\in I} \ind_{\{ z_i\leq a\}}
 \delta_{(z_i, \tilde \ct^i)}\right) \right] .
\end{multline*}
Then use the distribution of $A_h$ under $\N^\psi$ given in
Proposition 
\ref{prop:dbq} to conclude. 
\end{proof}


\begin{thebibliography}{10}

\bibitem{Abraham2007d}
{\sc Abraham, R., and Delmas, J.-F.}
\newblock {Fragmentation associated with L\'{e}vy processes using snake}.
\newblock {\em Probability Theory and Related Fields 141\/} (2007),
 113--154.

\bibitem{Abraham2009b}
{\sc Abraham, R., and Delmas, J.-F.}
\newblock {A continuum-tree-valued Markov process}.
\newblock {\em To appear in Ann. Probab.\/} (2011).

\bibitem{Abraham2011}
{\sc Abraham, R., and Delmas, J.-F.}
\newblock {Record process on the Continuum Random Tree}.
\newblock {\em ArXiv preprint arXiv:1107.3657\/} (2011).

\bibitem{Abraham2012}
{\sc Abraham, R., and Delmas, J.-F.}
\newblock {Record processes on L\'{e}vy trees}.
\newblock {\em In preparation\/} (2012).

\bibitem{Abraham2010}
{\sc Abraham, R., Delmas, J.-F., and He, H.}
\newblock {Pruning Galton-Watson Trees and Tree-valued Markov Processes}.
\newblock {\em arXiv preprint arXiv:1007.0370\/} (2010).

\bibitem{Abraham2012a}
{\sc Abraham, R., Delmas, J.-F., and Hoscheit, P.}
\newblock {A note on Gromov-Hausdorff-Prokhorov distance between (locally)
 compact measure spaces}.
\newblock {\em Preprint\/} (2012).

\bibitem{Abraham2010a}
{\sc Abraham, R., Delmas, J.-F., and Voisin, G.}
\newblock {Pruning a L\'{e}vy continuum random tree}.
\newblock {\em Electronic journal of probability 15\/} (2010), 1429--1473.

\bibitem{Aldous1991a}
{\sc Aldous, D.}
\newblock {The Continuum Random Tree. I}.
\newblock {\em The Annals of Probability 19}, 1 (1991), 1--28.

\bibitem{Aldous1998}
{\sc Aldous, D., and Pitman, J.}
\newblock {The standard additive coalescent}.
\newblock {\em The Annals of Probability 26}, 4 (1998), 1703--1726.

\bibitem{Aldous1998a}
{\sc Aldous, D., and Pitman, J.}
\newblock {Tree-valued Markov chains derived from Galton-Watson processes}.
\newblock In {\em Annales de l'Institut Henri Poincar\'{e} (B) Probability and
 Statistics\/} (1998), vol.~34, Elsevier, pp.~637--686.

\bibitem{Berestycki2009a}
{\sc Berestycki, J., Kyprianou, A., and Murillo-Salas, A.}
\newblock {The prolific backbone for supercritical superdiffusions}.
\newblock {\em Stochastic Processes and their Applications 121}, 6 (2011),
 1315--1331.

\bibitem{Burago2001}
{\sc Burago, D., Burago, Y.~D., and Ivanov, S.}
\newblock {\em {A course in metric geometry}}.
\newblock American Mathematical Society Providence, 2001.

\bibitem{Duquesne2002}
{\sc Duquesne, T., and {Le Gall}, J.-F.}
\newblock {\em {Random trees, L\'{e}vy processes and spatial branching
 processes}}, vol.~281 of {\em Ast\'{e}risque}.
\newblock SMF, 2002.

\bibitem{Duquesne2005a}
{\sc Duquesne, T., and {Le Gall}, J.-F.}
\newblock {Probabilistic and fractal aspects of L\'{e}vy trees}.
\newblock {\em Probability Theory and Related Fields 131}, 4 (2005), 553--603.

\bibitem{Duquesne2010a}
{\sc Duquesne, T., and Winkel, M.}
\newblock {General growth by hereditary properties and an invariance principle
 for Galton-Watson real trees}.
\newblock {\em Work in progress\/}.

\bibitem{Duquesne2007}
{\sc Duquesne, T., and Winkel, M.}
\newblock {Growth of L\'{e}vy trees}.
\newblock {\em Probability Theory and Related Fields 139}, 3-4 (2007),
 313--371.

\bibitem{Evans}
{\sc Evans, S.~N.}
\newblock {\em {Probability and real trees}}.
\newblock Lecture Notes in Mathematics. Springer, 2008.

\bibitem{Evans2005}
{\sc Evans, S.~N., Pitman, J., and Winter, A.}
\newblock {Rayleigh processes, real trees, and root growth with re-grafting}.
\newblock {\em Probability Theory and Related Fields 134}, 1 (2005), 81--126.

\bibitem{Evans2006}
{\sc Evans, S.~N., and Winter, A.}
\newblock {Subtree prune and regraft: A reversible real tree-valued Markov
 process}.
\newblock {\em The Annals of Probability 34}, 3 (2006), 918--961.

\bibitem{Greven2008b}
{\sc Greven, A., Pfaffelhuber, P., and Winter, A.}
\newblock {Convergence in distribution of random metric measure spaces
 ($\Lambda$-coalescent measure trees)}.
\newblock {\em Probability Theory and Related Fields 145}, 1-2 (2008),
 285--322.

\bibitem{Gromov1999metric}
{\sc Gromov, M.}
\newblock {\em {Metric Structures for Riemannian and Non-Riemannian Spaces}}.
\newblock Modern Birkh\"{a}user Classics. Birkh\"{a}user Boston, Boston, MA,
 2007.

\bibitem{Kyprianou2006}
{\sc Kyprianou, A.}
\newblock {\em {Introductory Lectures on Fluctuations of L\'{e}vy Processes
 with Applications}}.
\newblock Springer, Berlin, Heidelberg, 2006.

\bibitem{Lambert2008a}
{\sc Lambert, A.}
\newblock {Population Dynamics and Random Genealogies}.
\newblock {\em Stochastic Models 24\/} (2008), 45--163.

\bibitem{LeGall2006a}
{\sc {Le Gall}, J.-F.}
\newblock {Random real trees}.
\newblock In {\em Annales de la facult\'{e} des Sciences de Toulouse -
 Math\'{e}matiques\/} (2006), vol.~15, Citeseer, p.~35.

\bibitem{LeGall1998b}
{\sc {Le Gall}, J.-F., and {Le Jan}, Y.}
\newblock {Branching processes in L\'{e}vy processes: the exploration process}.
\newblock {\em The Annals of Probability 26}, 1 (1998), 213--252.

\bibitem{Li2011}
{\sc Li, Z.}
\newblock {\em {Measure-Valued Branching Markov Processes}}.
\newblock Probability and Its Applications. Springer Berlin Heidelberg, 2011.

\bibitem{Miermont2007}
{\sc Miermont, G.}
\newblock {Tessellations of random maps of arbitrary genus}.
\newblock {\em Ann. Sci. Ec. Norm. Sup\'{e}r. 42\/} (2009), 725--781.

\bibitem{Voisin2010}
{\sc Voisin, G.}
\newblock {Dislocation measure of the fragmentation of a general L\'{e}vy
 tree}.
\newblock {\em To appear in ESAIM: Probability and Statistics\/} (2011).

\end{thebibliography}
\end{document}